\documentclass[11pt,a4paper,leqno]{article}
\usepackage[arxiv]{mymacros}
\usepackage{calc}

\newcommand{\LL}{\bm{L}}
\newcommand{\PP}{\bm{P}}
\newcommand{\DD}{\bm{D}}
\newcommand{\TT}{\bm{T}}
\newcommand{\EE}{\bm{E}}

\newcommand{\uCdot}{\dot C}
\newcommand{\uQ}{Q}
\newcommand{\lCdot}{\mathsf{\dot C}}
\newcommand{\lQ}{\mathsf{Q}}
\newcommand{\lz}{\mathsf{z}}
\newcommand{\ludot}{\mathsf{\dot u}}
\newcommand{\lds}{\frac{ds}{\ell_s^2}}

\newcommand{\nV}[1]{\|\tilde V^{(#1)}\|}
\newcommand{\nVt}[1]{\|\tilde V_t^{(#1)}\|}
\newcommand{\nVs}[1]{\|\tilde V_s^{(#1)}\|}

\author{Roland Bauerschmidt\footnote{University of Cambridge, Statistical Laboratory, DPMMS. E-mail: {\tt rb812@cam.ac.uk}.} \and
  \and Thierry Bodineau\thanks{CMAP \'Ecole Polytechnique, CNRS, Universit\'e Paris-Saclay. E-mail: {\tt thierry.bodineau@polytechnique.edu}.}}
\title{Log-Sobolev inequality for the continuum sine-Gordon model}
\date{\vspace{-5ex}}

\begin{document}
\maketitle
\begin{abstract}
  We derive a multiscale generalisation of the Bakry--\'Emery criterion for a measure to satisfy a Log-Sobolev inequality.
  Our criterion relies on the control of an associated PDE well known in renormalisation theory: the Polchinski equation.
  It implies the usual Bakry--\'Emery criterion, 
  but we show that it remains effective for measures
  which are far from log-concave.
  Indeed, using our criterion,
  we prove that the massive continuum sine-Gordon model with $\beta < 6\pi$
  satisfies asymptotically optimal Log-Sobolev inequalities for Glauber and Kawasaki dynamics.
  These dynamics can be seen as singular SPDEs recently constructed
  via regularity structures, but our results are independent of this theory.
\end{abstract}

\section{Introduction and results}
\label{sec:intro}

\subsection{Introduction}

Log-Sobolev inequalities are strong inequalities with numerous general consequences,
including
concentration of measure,
relaxation and hypercontractivity of stochastic dynamics,
transport inequalities,
and others. See \cite{MR3155209,MR1837286} for a review.
They originate from Quantum Field Theory
where Log-Sobolev inequalities were first derived for Gaussian measures
as a tool to study non-Gaussian measures in infinite dimensions
(Euclidean Quantum Field Theories, EQFTs)
\cite{MR0210416,doi:10.1063/1.1664760,MR0420249}.
As a consequence of a general new approach, we prove Log-Sobolev inequalities for
the massive sine-Gordon model. This is a fundamental example of a \emph{non-Gaussian} EQFT in two dimensions
and its stochastic dynamics is a prototypical example of a singular SPDE.

As Log-Sobolev inequalities provide strong control on the measures they apply to,
proving them remains in general a difficult problem even if the equilibrium correlation
functions are well understood.
This applies especially to strongly correlated measures.
For log-concave measures (or measures satisfying a curvature dimension condition),
the fundamental Bakry--\'Emery criterion 
provides a simple and often quite sharp sufficient condition \cite{MR889476,MR2213477}.
In its proof, a Log-Sobolev inequality for
a Markov semigroup is derived by integration of local Log-Sobolev inequalities for the
\emph{same} Markov semigroup.
Our method also uses local Log-Sobolev inequalities, but for a semigroup that is \emph{different}
from the one for which the Log-Sobolev inequality is proven.
Namely our method uses the time-dependent semigroup driven by the Polchinski equation,
a version of the renormalisation semigroup.
Unlike the original semigroup, this Polchinski semigroup provides a notion of scale
and hence we effectively
obtain a multiscale version of the Bakry--\'Emery criterion.

The simplest version of our new Polchinski equation criterion for the Log-Sobolev inequality is stated in Section~\ref{sec:intro-lsi}.
In Example~\ref{ex:BE}, we illustrate that it implies the Bakry--\'Emery criterion.
As an application of the new criterion,
demonstrating that it remains effective for measures that are far from log-concave,
we prove the following theorem for the continuum sine-Gordon model.
For a precise statement of this result and related discussion, we refer to Section~\ref{sec:intro-sg}.
In Section~\ref{sec:intro-discuss}, we discussed further directions and related results.

\begin{theorem}
  The continuum massive sine-Gordon model with $\beta < 6\pi$ 
  satisfies asymptotically optimal Log-Sobolev inequalities for Glauber and Kawasaki dynamics
  (under suitable conditions).
\end{theorem}

Throughout this paper, we make the assumption that all functions considered are Borel measurable
and that all functions to which derivatives are applied are continuously differentiable of the required order.

\subsection{Polchinski equation and Log-Sobolev inequality}
\label{sec:intro-lsi}

In this section we state the simplest version of our new criterion
for a probability measure to satisfy a Log-Sobolev inequality.
  
Given a linear space $X \subseteq \R^N$ with the induced inner product $(\cdot,\cdot)$,
a symmetric matrix $A$ that acts positive definitely on $X$, and a potential $V_0: X \to \R$,
we consider the probability measure $\nu_0$ with expectation
\begin{equation} \label{e:nu0-A}
  \E_{\nu_0}F \propto \int_X e^{-\frac12(\zeta,A\zeta)-V_0(\zeta)} \, F(\zeta)\, d\zeta.
\end{equation}
We call the set $\Lambda= \{1,\dots,N\}$ the \emph{index space} and the space $X$ the \emph{field space};
see also Figure~\ref{fig:LambdaX}.
Let $Q_t = e^{-tA/2}$ be the \emph{heat semigroup} associated with $A$
(acting on elements $\varphi \in X$, i.e., functions $\varphi: \Lambda \to \R$
on the index space), set
\begin{equation}
  \dot C_t = Q_t^2 = e^{-tA}, \qquad C_t = \int_0^t \dot C_s \, ds,
\end{equation}
and denote by $\EE_{C_s}$ the expectation of the Gaussian measure with covariance $C_s$.
For $t>s>0$, we define the \emph{renormalised potential} $V_t$,
the \emph{renormalisation semigroup} $\PP_{s,t}$
(acting on functions $F:X \to \R$ on the field space),
and the \emph{renormalised measure} $\nu_t$ by
\begin{align} \label{e:V-def}
  e^{-V_t(\varphi)} &= \EE_{C_t}(e^{-V_0(\varphi+\zeta)}),
  \\
  \label{e:P-def}
  \PP_{s,t}F(\varphi) &= e^{V_t(\varphi)} \EE_{C_t-C_s}(e^{-V_s(\varphi+\zeta)} F(\varphi+\zeta)),
  \\
  \label{e:nu-def}
  \E_{\nu_t} F = \PP_{t,\infty}F(0) &= e^{V_\infty(0)} \EE_{C_\infty-C_t} (e^{-V_t(\zeta)} F(\zeta)),
\end{align}
where $\varphi \in X$, the expectation $\EE_{C_t}$ applies to $\zeta$,
and it is natural to define $\E_{\nu_\infty} F = F(0)$.
Essentially equivalently to \eqref{e:V-def}, $V_t$ solves the Polchinski equation; see \eqref{e:polchinski} below.

In what follows, we will impose the following \emph{ergodicity assumption} on the semigroup $\PP$:
For all bounded smooth functions $F: X \to \R$ and $g:\R \to \R$,
\begin{equation} \label{e:ergodicity}
  \E_{\nu_t}g(\PP_{0,t}F) \to g(\E_{\nu_0}F) \quad \text{as $t\to\infty$.}
\end{equation}
Like the ergodicity assumption in the Bakry--\'Emery theory (see \cite{MR3155209,MR1845806}),
this assumption is \emph{qualitative} and easily seen to be satisfied in all examples of interest. 

The following theorem bounds the Log-Sobolev constant of the measure $\nu_0$.
For its statement,
recall that the relative entropy of $F: X \to \R_+$ with respect to $\nu_0$ is
given by
\begin{equation}
  \ent_{\nu_0}(F) = \E_{\nu_0} \Phi(F) - \Phi(\E_{\nu_0} F), \qquad \Phi(x) = x\log x,
\end{equation}
where $0\log 0 = 0$.
We write $\nabla$ for the gradient on $X$ and $(\nabla F)^2 = (\nabla F, \nabla F)$;
thus in particular if $X=\R^N$ then $(\nabla F)^2 = \sum_{i=1}^N (\ddp{F}{\varphi_i})^2$.

\begin{figure}
  \begin{center}
    \def\svgscale{0.9}
%% Accompanies image file 'grid.pdf' (pdf, eps, ps)
\begingroup%
  \makeatletter%
  \providecommand\color[2][]{%
    \errmessage{(Inkscape) Color is used for the text in Inkscape, but the package 'color.sty' is not loaded}%
    \renewcommand\color[2][]{}%
  }%
  \providecommand\transparent[1]{%
    \errmessage{(Inkscape) Transparency is used (non-zero) for the text in Inkscape, but the package 'transparent.sty' is not loaded}%
    \renewcommand\transparent[1]{}%
  }%
  \providecommand\rotatebox[2]{#2}%
  \newcommand*\fsize{\dimexpr\f@size pt\relax}%
  \newcommand*\lineheight[1]{\fontsize{\fsize}{#1\fsize}\selectfont}%
  \ifx\svgwidth\undefined%
    \setlength{\unitlength}{396.23125186bp}%
    \ifx\svgscale\undefined%
      \relax%
    \else%
      \setlength{\unitlength}{\unitlength * \real{\svgscale}}%
    \fi%
  \else%
    \setlength{\unitlength}{\svgwidth}%
  \fi%
  \global\let\svgwidth\undefined%
  \global\let\svgscale\undefined%
  \makeatother%
  \begin{picture}(1,0.3217818)%
    \lineheight{1}%
    \setlength\tabcolsep{0pt}%
    \put(0,0){\includegraphics[width=\unitlength,page=1]{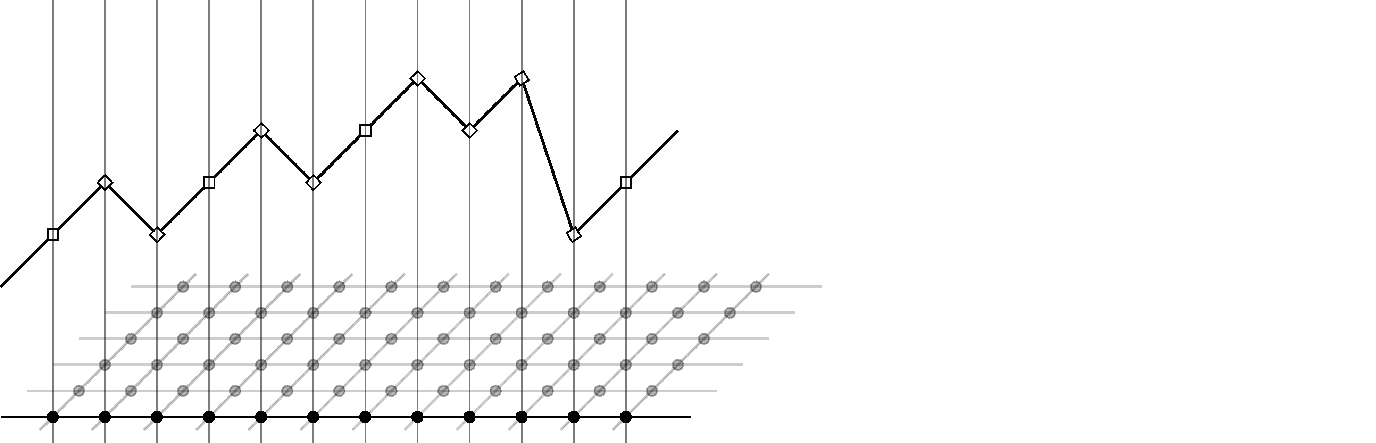}}%
    \put(0.52397164,0.0312606){\color[rgb]{0,0,0}\makebox(0,0)[lt]{\small\smash{\begin{tabular}[t]{l}index space $\Lambda$\end{tabular}}}}%
    \put(0.46995616,0.2659528){\color[rgb]{0,0,0}\makebox(0,0)[lt]{\small\smash{\begin{tabular}[t]{l}field space $X$\end{tabular}}}}%
    \put(0,0){\includegraphics[width=\unitlength,page=2]{grid}}%
    \put(0.80262337,0.26465877){\color[rgb]{0,0,0}\makebox(0,0)[lt]{\scriptsize\smash{\begin{tabular}[t]{l}original\\semigroup\end{tabular}}}}%
    \put(0.88097328,0.16183732){\color[rgb]{0,0,0}\makebox(0,0)[lt]{\scriptsize\smash{\begin{tabular}[t]{l}Polchinski\\semigroup\end{tabular}}}}%
    \put(0.80346621,0.06774123){\color[rgb]{0,0,0}\makebox(0,0)[lt]{\scriptsize\smash{\begin{tabular}[t]{l}heat semigroup\end{tabular}}}}%
  \end{picture}%
\endgroup%

    \caption{The heat semigroup $Q_t$ acts on the index space $\Lambda = \{1,\dots,N\}$, i.e., `horizontally.'
      In our primary applications, the index space $\Lambda$ is identified with a finite approximation
      to $\Z^d$ or $\R^d$ and $A$ is the Laplacian on $\Lambda$.
      The original semigroup with Dirichlet form $\E_{\nu_0}(\nabla F)^2$ acts on the field space $X \subseteq \R^\Lambda$.
      It acts `vertically' in the sense that the principal part of its generator is the
      standard Laplacian on $X$, i.e., $\Delta_{\id}$ in the notation \eqref{e:Deltawdef}.
      The Polchinski renormalisation semigroup $\PP_{s,t}$ also acts on field space $X$, but it acts `diagonally'
      in the sense that the principal part of its generator is time dependent and given in terms of the heat kernel
      as $\Delta_{Q_t^2}$ (see \eqref{e:polchinski-L}). 
    \label{fig:LambdaX}}
\end{center}
\end{figure}

\begin{theorem} \label{thm:LSI-heat}
  In the set-up above, assume \eqref{e:ergodicity},
  let $\lambda > 0$ be the smallest eigenvalue of $A$,
  suppose there are real numbers $\dot\mu_t$ (possibly negative) such that for all $t \geq 0$,
  as quadratic forms on $X$,
  \begin{equation} \label{e:assVt-heat}
    Q_t\He V_t(\varphi) Q_t\geq \dot\mu_t \id,
    \qquad \text{where $Q_t = e^{-tA/2}$,}
  \end{equation}
  and define $\mu_t = \int_0^t \dot\mu_s \, ds$.
  Then $\nu_0$ satisfies the Log-Sobolev inequality
  \begin{equation} \label{e:gammadef-heat}
    \ent_{\nu_0}(F)
    \leq \frac{2}{\gamma} \E_{\nu_0}(\nabla \sqrt{F})^2,
    \qquad
    \frac{1}{\gamma} = \int_0^\infty e^{-\lambda t-2\mu_t} \, dt,
  \end{equation}
  provided the integral is finite.
\end{theorem}

The proof of Theorem~\ref{thm:LSI-heat}, given in Section~\ref{sec:LSI},
shares significant elements with the celebrated Bakry--\'Emery argument,
but with the crucial difference that it uses the time-dependent Polchinski semigroup \eqref{e:P-def}
rather than the original semigroup, associated with the Dirichlet form
$\E_{\nu_0}(\nabla F)^2$, to decompose the relative entropy.
The above version of our criterion relies on the particular decomposition of the matrix $C_\infty=A^{-1}$ in terms of the heat semigroup $\dot C_t = e^{-tA}$.
In Section~\ref{sec:LSI}, we also consider variations of the criterion that apply to other decompositions.

To apply the theorem, the main task is to verify \eqref{e:assVt-heat}.
It is not difficult to see that the renormalised potential $V_t$ solves
the \emph{Polchinski equation}
(see Section~\ref{sec:intro-discuss} for its history)
\begin{equation}
  \label{e:polchinski}
  \partial_t V_t = \frac12 \Delta_{\dot C_t} V_t - \frac12 (\nabla V_t)_{\dot C_t}^2,
\end{equation}
where we use the notation (and with $w=\id$ if the argument $w$ is omitted)
\begin{equation} \label{e:Deltawdef}
  (u,v)_w = \sum_{i,j}w_{ij} u_iv_j,
  \quad
  (\nabla F)_{w}^{2}= (\nabla F,\nabla F)_w, %\sum_{i,j}w_{ij}(\partial_{i} f) (\partial_{j} f),
  \quad
  \Delta_w F = (\nabla ,\nabla )_w F. %  = \sum_{i,j} w_{ij} \partial_i \partial_j f.
\end{equation}
In general, verifying \eqref{e:assVt-heat}
is a challenging problem because the Polchinski equation is a non-linear PDE
in $N$ dimensions, where in the examples of main interest $N\to\infty$.
Nonetheless, we believe that the required estimates are true in many relevant examples,
including spin systems near the critical point.
In particular, in Section~\ref{sec:SG}, we verify the condition for the continuum sine-Gordon model
by analysing the Polchinski equation.

To illustrate our new criterion,
we note briefly that \eqref{e:assVt-heat} is not hard to verify for log-concave measures,
in which case we recover 
the Bakry--\'Emery criterion as a special case.

\begin{example}[Bakry--\'Emery criterion] \label{ex:BE}
Consider a probability measure $\nu_0$ with expectation
\begin{equation} \label{e:measureH}
  \E_{\nu_0} F \propto \int_{\R^N} e^{-H(\zeta)} F(\zeta) \, d\zeta,
\end{equation}
where $\He H \geq \lambda \id$ holds uniformly for some $\lambda>0$.
Equivalently, $\nu_0$ can be written as in \eqref{e:nu0-A}:
\begin{equation}
  H(\zeta) = \frac12 (\zeta,A\zeta) + V_0(\zeta),
  \qquad
  \text{with $A=\lambda \id$ and $V_0$ convex.}
\end{equation}
It follows that $V_t$ is convex for all $t\geq 0$ (see, e.g., \cite[Theorem~4.3]{MR0450480}).
Hence $\mu_t \geq 0$ for all $t$ and thus $\gamma \geq \lambda$ in \eqref{e:gammadef-heat}.
This is the Bakry--\'Emery criterion.

We remark that an alternative proof that $V_t$ remains convex for $t>0$ can be deduced
from the maximum principle for symmetric tensors \cite[Theorem~9.1]{MR664497}.
This argument is in fact analogous to the proof
that positive Ricci curvature remains positive under the Ricci flow
in \cite{MR664497}.
\end{example}

Theorem~\ref{thm:LSI-heat} can be considered a multiscale version of the Bakry--\'Emery criterion
in which the global convexity assumption $\inf_\varphi \He V_0(\varphi) \geq 0$,
which is equivalent to $\inf_{t\geq 0}\inf_\varphi \He V_t(\varphi) \geq 0$, is replaced by the
assumption \eqref{e:assVt-heat} on the Hessians of the effective potential $V_t$ at each scale $t$.
We emphasise that these Hessians are \emph{not} required be positive definite;
and in fact in the example of the continuum sine-Gordon model which we consider in Section~\ref{sec:intro-sg} below,
the effective potential remains non-convex at all scales $t>0$.
We also emphasise that the application of the heat kernel $Q_t$ to $\He V_t(\varphi)$ in \eqref{e:assVt-heat}
has an important smoothing effect.
In particular,
for the sine-Gordon model, we will see that this smoothing effect is essential when $\beta>4\pi$.

\begin{remark} \label{rk:uniqueness}
We have defined the renormalised potential $V_t$ as the convolution solution \eqref{e:V-def} to the Polchinski equation \eqref{e:polchinski}.
Since equivalently $Z_t = e^{-V_t}$ solves the heat equation $\partial_tZ_t = \frac12\Delta_{\dot C_t}Z_t$,
the Polchinski equation has a unique solution under weak conditions.
Then one may equivalently solve \eqref{e:polchinski} instead of \eqref{e:V-def};
for an example for which this is useful, see Section~\ref{sec:SG}.
\end{remark}

\begin{remark} \label{rk:metric}
We remark that with the time-dependent metric $g_t = e^{+tA}$ on $X$ and $\nabla_{g_t}$ and $\Delta_{g_t}$ defined
as in Riemannian geometry, i.e., $\nabla_{g_t} = g_t^{-1} \nabla$ and $\Delta_{g_t}$ the Laplace-Beltrami operator,
one has $\Delta_{\dot C_t} = \Delta_{g_t}$ and $(\nabla F)_{\dot C_t}^2 = (\nabla_{g_t} F)_{g_t}^2$.
The condition \eqref{e:assVt-heat} then becomes
$\He_{g_t} V_t \geq \dot \mu_t g_t$.
\end{remark}

\subsection{Continuum sine-Gordon model}
\label{sec:intro-sg}

In Section~\ref{sec:SG}, we apply Theorem~\ref{thm:LSI-heat} to prove asymptotically sharp
Log-Sobolev inequalities for
Glauber and Kawasaki dynamics of the massive continuum sine-Gordon model with $\beta < 6\pi$.
The massive sine-Gordon model is a fundamental example of a two-dimensional interacting Euclidean Quantum Field Theory,
i.e., a non-Gaussian probability measure on $\cD'(\R^2)$ sometimes \emph{formally} written as
\begin{equation}
  \frac{1}{Z}
  \exp\qa{-
    \int_{\R^2} \pa{ \frac12 \varphi(-\Delta\varphi) + \frac12 m^2\varphi(x)^2 
    + 2z : \cos(\sqrt{\beta}\varphi(x)) : }  \, dx} \prod_{x\in\R^2} d\varphi(x).
\end{equation}
Here $\Delta$ is the Laplacian on $\R^2$,
and the notation $:$ denotes Wick ordering, i.e., that $z$ is formally multiplied
by a divergent constant (making the microscopic potential extremely non-convex);
see \eqref{e:sg-Heps}-\eqref{e:sg-zeps} below for the precise definition that we will use.
The Glauber dynamics of the sine-Gordon model (also called dynamical sine-Gordon model)
can be realised as a singular SPDE that was recently constructed
using the theory of regularity structures.
References on the sine-Gordon model are provided further below.

For clarity, we consider the model in a lattice approximation of a two-dimensional torus,
and prove estimates uniformly in the lattice spacing and in the size of the torus.
Therefore, from now on, let $d=2$, let $\Omega_L=L \T^d$ be the torus of side length $L>0$,
and let $\Omega_{\epsilon,L} = \Omega_L \cap \epsilon \Z^d$ be its lattice approximation with mesh size $\epsilon>0$
(where we always assume $L$ is a multiple of $\epsilon$).
The continuum sine-Gordon model $\nu_{\epsilon,L}$ in the lattice approximation
is the probability measure on $\R^{\Omega_{\epsilon,L}}$
with density proportional to $e^{-H_{\epsilon,L}(\varphi)}$
where $H_{\epsilon,L}$ is defined for $\varphi : \Omega_{\epsilon,L} \to \R$ by
\begin{equation} \label{e:sg-Heps}
  H_{\epsilon,L}(\varphi)
  = {\epsilon^d}\sum_{x\in \Omega_{\epsilon,L}} \pbb{ \frac12 \varphi_x(-\Delta^\epsilon\varphi)_x
  + \frac{1}{2}m^2 \varphi_x^2
  + 2z_\epsilon \cos(\sqrt{\beta}\varphi_x)
  },
\end{equation}
with \emph{divergent} coupling constant
\begin{equation} \label{e:sg-zeps}
   z_\epsilon = z \epsilon^{-\beta/4\pi},
\end{equation}
and where
$(\Delta^\epsilon \varphi)_x = \epsilon^{-2} \sum_{y\sim x} (\varphi_y-\varphi_x)$
is the discretised Laplacian, i.e., the sum $y \sim x$
is over nearest neighbour vertices $y$ of $x$ in $\epsilon\Z^d$.
Under suitable assumptions,
this normalisation ensures that, for $0<\beta<8\pi$,
the measures $\nu_{\epsilon,L}$ converge weakly to a non-Gaussian probability measure
on $\cD'(\R^2)$ as $\epsilon\to 0$ and $L\to\infty$; see the discussion after the statement
of the theorems below.

Our first theorem is a uniform Log-Sobolev inequality for the Glauber dynamics
of the massive  sine-Gordon measure $\nu_{\epsilon,L}$
(with dimension always $d=2$). The Glauber Dirichlet form is given by
\begin{equation}   \label{e:Dirichlet-eps}
  \DD_{\epsilon,L}(F) = \frac{1}{\epsilon^2} \sum_{x\in \Omega_{\epsilon,L}}
  \E_{\nu_{\epsilon,L}} \qa{ \pa{\ddp{F}{\varphi_x}}^2 },
\end{equation}
corresponding to the system of SDEs
\begin{equation} \label{e:sg-SDE}
  \ddp{}{t}\varphi_x^\epsilon
  = (\Delta^\epsilon\varphi^\epsilon )_x
  + m^2 \varphi^\epsilon_x
  + \epsilon^{-\beta/4\pi}2z\sqrt{\beta} \sin(\sqrt{\beta}\varphi_x^\epsilon)
  + \sqrt{2} \dot W_x^\epsilon,
\end{equation}
where
$\dot W^\epsilon$ is space-time white noise (with discretised space), i.e.,
the $(W_x^\epsilon)_{x\in\Omega_{\epsilon,L}}$ are independent Brownian motions with quadratic variation $\avg{W_x^\epsilon}(t) = t/\epsilon^2$.

\begin{theorem} \label{thm:sg-4pi}
  Fix $\beta < 6\pi$, and let $L>0$, $m>0$, and $z\in\R$. Then there
  is $\gamma(\beta,z,m,L) > 0$ independent of $\epsilon>0$ such that, for all $F\geq 0$,
  \begin{equation}
    \ent_{\nu_{\epsilon,L}}(F) \leq \frac{2}{\gamma(\beta,z,m,L)} \DD_{\epsilon,L}(\sqrt{F}).
  \end{equation}
  Moreover, there is $\delta_\beta>0$ such that if $Lm \geq 1$ and $|z|m^{-2+\beta/4\pi} \leq \delta_\beta$, then
  \begin{equation} \label{e:sg-asymp}
    \gamma(\beta,z,m,L)
    \geq
    m^2 - O_\beta(m^{\beta/4\pi}|z|)
    ,
  \end{equation}
  where the constant $O_\beta$ depends on $\beta$ only (and is thus uniform in $L \geq 1/m$).  
\end{theorem}

Our next theorem is a (conservative) Kawasaki version of the previous result.
We thus consider the measure $\nu_{\epsilon,L}^{0}$ obtained by constraining the mean
spin of the measure $\nu_{\epsilon,L}$ to  $\sum_{x\in\Omega_{\epsilon,L}}\varphi_x=0$,
i.e., $\nu_{\epsilon,L}^{0}$ is supported on $\{\varphi : \sum_x \varphi_x=0\}$.
(The same proof also works for arbitrary nonzero mean of $\varphi$.)
The Dirichlet form for Kawasaki dynamics with invariant measure $\nu_{\epsilon,L}^0$ is defined by
\begin{equation}
  \label{e:Dirichlet-Kawa-eps}
  \DD_{\epsilon,L}^0(F) = \frac{1}{\epsilon^{4}} \sum_{x\sim y\in \Omega_{\epsilon,L}}
  \E_{\nu^0_{\epsilon,L}} \qa{ \pa{\ddp{F}{\varphi_x}-\ddp{F}{\varphi_y}}^2 }.
\end{equation}

\begin{theorem} \label{thm:sg-4pi-kawasaki}
  Fix $\beta < 6\pi$, and let $L>0$, $m>0$, and $z\in\R$. Then
  there is $\gamma^0(\beta,z,m,L) > 0$ independent of $\epsilon>0$ such that,
  for all $F\geq 0$,
  \begin{equation}
    \ent_{\nu_{\epsilon,L}^0}(F) \leq \frac{2}{\gamma^0(\beta,z,m,L)} \DD_{\epsilon,L}^0(\sqrt{F}).
  \end{equation}
  Moreover, there is $\delta_\beta>0$ such that if $Lm \geq 1$ and $|z|m^{-2+\beta/4\pi} \leq \delta_\beta$, then
  \begin{equation}
    \gamma^0(\beta,z,m,L) \geq
    \frac{(2\pi)^2}{L^2} \pa{m^2 + \frac{(2\pi)^2}{L^2} - O_\beta(m^{\beta/4\pi}|z|)}
    ,
  \end{equation}
  where the constant $O_\beta$ depends on $\beta$ only (and is thus uniform in $L \geq 1/m$).
\end{theorem}

For $z=0$, the sine-Gordon model degenerates simply to the continuum Gaussian free field
with covariance $(-\Delta+m^2)^{-1}$,
as $\epsilon \downarrow 0$, for which the Glauber Log-Sobolev constant is $m^2$
(by \cite{MR0420249} or the Bakry--\'Emery criterion), and similarly in the Kawasaki case.
Note that, in this scaling in which the convexity of the Gaussian measure is of order $1$,
the best lower bound on the Hessian of the interaction term $V_0$ is of order $-\epsilon^{-\beta/4\pi}$ if $z\neq 0$
and thus tends to $-\infty$ as $\epsilon \to 0$.
Thus the measure is far out of the scope of the applicability of the Bakry--\'Emery criterion if $z\neq 0$.
Our proof of the above theorems via Theorem~\ref{thm:LSI-heat} relies on the smoothing of the
effective potential $V_t$ along the flow of the Polchinski equation.

The Glauber dynamics of the sine-Gordon model is considered in \cite{MR3452276,1808.02594}.
Using the theory of regularity structures, it is shown in these references that versions of \eqref{e:sg-SDE}
that are regularised in space-time instead of space only
converge as $\epsilon\to 0$ pathwise in a space of distributions
on a short noise-dependent time interval.
In our setting, it is essential that the noise is white in time for the regularised dynamics to define a Markov process.
The question of regularisation in space rather than space-time
was considered for the closely related problems of the subcritical
continuum $\varphi^4$ model and KPZ equation in \cite{MR3592748,MR3785597,MR3758734}
as well as in \cite{MR3628883,MR4029153,1808.09429}.
Presumably similar arguments would apply also to the sine-Gordon model, but have not been carried out.

Finally, we provide some references on the continuum sine-Gordon model.
For $0<\beta<8\pi$, at least when the domain is a torus and $z \neq 0$ is small and $m^2>0$,
it is known that $\nu_{\epsilon,L} \to \nu$ weakly, where $\nu$ is a non-Gaussian measure on $\cD'(\R^2)$
with a precise description in terms of renormalised expansions;
see
\cite{MR0421409,MR0443693},
\cite{MR649810,MR849210}, \cite{MR914427}, and \cite{MR1777310,MR1638615,MR1240586}
for different approaches.
This result is simplest for $\beta<4\pi$,
when in finite volume the continuum sine-Gordon measure is absolutely continuous with respect to the Gaussian free field.
For $4\pi \leq \beta < 8\pi$, there is an infinite sequence of thresholds at
$\beta=8\pi(1-1/2n)$, $n=1,2,\dots$,
at which the partition function (but not the normalised probability measure) acquires divergent contributions;
see \cite{MR649810} for further discussion.
The physical meaning of these divergences remains debated \cite{10.1007/BF02179380}.
The sine-Gordon model satisfies a very interesting duality with the massive Thirring model,
the Coleman correspondence or Bosonization \cite{PhysRevD.11.2088}. For restricted values of $\beta$,
this correspondence has been established in finite volume or with a mass term \cite{MR0443693,MR1672504,MR2461991},
but in general its proof remains an open problem, most importantly in the formally massless case $m^2=0$.
In particular, under this correspondence, for the special value $\beta=4\pi$, the correlations functions of the
sine-Gordon model are equivalent to those of free fermions.
In general, an important question for the sine-Gordon model that has remained open
is the formally massless case $L \to \infty$ and $m^2\to 0$,
in which case correlations decay polynomially if $z=0$.
For $z\neq 0$, it is conjectured that the equilibrium correlation functions
have exponential decay, for any $\beta<8\pi$.
Closely related results for small $\beta$ were obtained in \cite{MR574172,MR923850}.
It would be very interesting to understand the dynamical behaviour in this regime.

Our result extends up to the second threshold $\beta< 6\pi$ and makes use of the approach of \cite{MR914427}.
It remains a very interesting problem to extend our results to the optimal regime $\beta<8\pi$.
Recent progress in the direction of extending the method of \cite{MR914427} includes \cite{MR4010781}.
Other recent results for the sine-Gordon model include \cite{1806.02118}.
For a one-dimensional analogue of the sine-Gordon model, a recent construction using martingales was given in \cite{1903.01394}.

\subsection{More discussion of our approach and of further directions}
\label{sec:intro-discuss}

Our approach to the Log-Sobolev inequality
involves the Polchinski equation \eqref{e:polchinski}.
The Polchinski equation 
is a continuous version of Wilson's renormalisation group
(which typically proceeds in discrete steps)
and variations of it go back to  \cite{PhysRevB.4.3174,PhysRevA.8.401},
while the continuous point of view was first systematically used
by Polchinski \cite{polchinski1984269}.
See \cite{MR2279216} for a review of its history
as well as for an account of the important role it has played in recent
advances in Perturbative Quantum Field Theory.
The relation of the Polchinski equation to the Mayer expansion
and its iterated versions was investigated in \cite{MR914427}
on which we rely for the sine-Gordon model.
Ideas related to the Polchinski equation were also used recently
in \cite{1805.10814} for a simple construction of the continuum $\varphi^4$ model in $d=2,3$.
We also mention that approaches involving aspects of renormalisation have
been used for a long time to study dynamics of spin systems, e.g., in the form of block dynamics
\cite{MR1746301,MR1931585,MR1414837}
and more recently in the two-scale approach \cite{MR3098070,MR2521405,MR2291434,1807.09857}.
Our approach involves infinitely many scales.

The regime of the continuum limit considered in Section~\ref{sec:intro-sg}
is known as the \emph{ultraviolet problem} in physics,
which for the two-dimensional sine-Gordon model is well-posed for $\beta<8\pi$.
The long-distance behaviour is predicted to be independent of $\epsilon$.
For $\beta<8\pi$, it can studied as a property of the continuum limit $\epsilon \to 0$,
but it makes
sense for all $\beta>0$ when the regularisation $\epsilon$ is fixed (lattice problem).
For $\beta \geq \beta_c$ (where the curve $\beta_c(z)$ passes through $8\pi$ at $z=0$, see \cite{MR2917175,1311.2237})
and small $z$ and $m^2=0$,
the scaling limit is known to be Gaussian free field in a suitable sense, for the model defined on the torus
\cite{1311.2237,MR1101688}.
This is called the \emph{infrared problem} in physics.
However, we emphasise that,
while the ultraviolet problem can be translated to a lattice problem, as we do,
the scaling of the infrared problem is more delicate than that of the ultraviolet problem.
For the sine-Gordon model, in the ultraviolet limit, the microscopic coupling constant is very small, of order $\epsilon^{2-\beta/4\pi} \ll 1$.
For the infrared problem, the microscopic coupling constant is of order $1$,
and unlikely field configurations play a more important role in understanding the measure
(large field problem); see \cite{MR1101688,MR2917175,1311.2237}.
We studied the spectral gap for the hierarchical version of the infrared problem in \cite{MR4061408}.
Using Theorem~\ref{thm:LSI-asy} and the estimates proved  in \cite{MR4061408}, the results for the
spectral gap stated in \cite{MR4061408} can be improved to results about the Log-Sobolev constant; see Example~\ref{ex:hierarchical}.

The next natural class of models that would be interesting to apply Theorem~\ref{thm:LSI-heat} to is
the $\varphi^4$ model.
The problem analogous of the one considered for the sine-Gordon model would be
the continuum $\varphi^4$ model on $\R^d$ where $d=2,3$ with sufficiently large mass (ultraviolet problem).
On a finite two-dimensional torus, a spectral gap result for the continuum $\varphi^4$ model has been shown in \cite{MR3825880}.
We stress again that the Polchinski equation has also been used in \cite{1805.10814}
in the construction of the continuum $\varphi^4$ model on a torus in $d=2,3$.
As in the case of the sine-Gordon model, the infrared problem appears more difficult than the ultraviolet problem.
For the latter we expect that
the Log-Sobolev constant
of the lattice $\varphi^4$ model or the Ising model in $d = 4$ (respectively $d>4$)
scales as $u(-\log u)^{z}$ (respectively $u$) as the critical point is approached with distance $u \downarrow 0$.
Again, for the hierarchical $\varphi^4$ model, we proved
the analogous statement for the spectral gap in \cite{MR4061408}
and the results of this paper can again be used to improve the latter result
to prove also an analogous Log-Sobolev inequality;
again see Example~\ref{ex:hierarchical}.

In a different direction, 
the Bakry--\'Emery theory has a well-known formulation in the context of manifolds (and beyond).
The Polchinski equation is closely related to the Gaussian convolution semigroup $\EE_{C_t}$ on $X$
and thus to the linear structure of $X$.
However with the disintegration of the Gaussian measure taking the role of
the reverse Ricci flow, there is an interesting resemblence
of our construction with those in \cite{0211159,MR2666905,MR2507614};
see also Remark~\ref{rk:metric}.

Finally, we remark that Log-Sobolev inequalities are a very useful tool to derive mixing results in general,
see, e.g., \cite{MR3020173}. It would be very interesting to derive such results in our context.

\section{Log-Sobolev inequality and the Polchinski equation}
\label{sec:LSI}

In this section we prove Theorem~\ref{thm:LSI-heat}
and variations of this result that apply in slightly different set-ups.
The proofs share many elements with the Bakry--\'Emery argument
which we will review.

\subsection{The renormalisation semigroup}
\label{sec:LSI-setup}

Let $t\in [0,\infty] \mapsto C_t$ be a function of positive semidefinite matrices on $\R^N$ increasing continuously as quadratic forms to a matrix $C_\infty$.
More precisely, we assume that $C_t = \int_0^t \dot C_s \, ds$ for all $t$,
where $t\mapsto \dot C_t$ is a bounded function with values in the space of positive semidefinite matrices
that is the derivative of $C_t$ except at isolated points.
As before, we denote by $\EE_{C_t}$ the expectation of the possibly degenerate Gaussian measure with covariance $C_t$.
We consider a probability measure $\nu_0$ with expectation 
\begin{equation} \label{e:nu0-Cinfty}
  \E_{\nu_0}F
  \propto
  \EE_{C_\infty} (e^{- V_0(\zeta)} F(\zeta))
  ,
\end{equation}
for a potential $V_0:\R^N \to \R$.
For $t>s>0$, we recall the definitions
\begin{align} \label{e:V-def-bis}
  e^{-V_t(\varphi)} &= \EE_{C_t}(e^{-V_0(\varphi+\zeta)}), 
  \\
  \label{e:P-def-bis}
  \PP_{s,t}F(\varphi) &= e^{V_t(\varphi)} \EE_{C_t-C_s}(e^{-V_s(\varphi+\zeta)} F(\varphi+\zeta)),
  \\
  \label{e:nu-def-bis}
  \E_{\nu_t} F = \PP_{t,\infty}F(0) &= e^{V_\infty(0)} \EE_{C_\infty-C_t} (e^{-V_t(\zeta)} F(\zeta)),
\end{align}
where the expectations again apply to $\zeta$.
We impose the following \emph{continuity assumption}:
For all bounded smooth functions $F: X \to \R$ and $g:\R \to \R$,
\begin{equation} \label{e:continuity}
  \E_{\nu_t}g(\PP_{0,t}F) \quad \text{is continuous in $t\in[0,+\infty].$}
\end{equation}
The assumption \eqref{e:continuity} reduces to \eqref{e:ergodicity}
when $C_t$ is differentiable in $t$, as in Section~\ref{sec:intro-lsi},
and it is again clear in all examples of practical interest.

The following proposition collects some properties of the above definitions;
we postpone its elementary proof to Section~\ref{sec:proof1}.

\begin{proposition} \label{prop:polchinski}
  Let $(C_t)$ be as above, let $V_0 \in C^2$, and assume \eqref{e:continuity}.
  Then for every $t$ such that $C_t$ is differentiable
  the renormalised potential $V$ defined in \eqref{e:V-def} satisfies the Polchinski equation
  \begin{align}
    \label{e:polchinski-bis}
    \partial_t V_t &= \frac12 \Delta_{\dot C_t} V_t - \frac12 (\nabla V_t)_{\dot C_t}^2.
  \end{align}
  The operators $(\PP_{s,t})_{s\leq t}$ form a time-dependent Markov semigroup with generators $(\LL_t)$,
  in the sense that
  $\PP_{t,t} = \id$ and $\PP_{r,t}\PP_{s,r}=\PP_{s,t}$ for all $s\leq r \leq t$,
  that $\PP_{s,t}F \geq 0$ if $F \geq 0$ and $\PP_{s,t}1=1$,
  and that for all $t$ at which $C_t$ is differentiable
  (respectively $s$ at which $C_s$ is differentiable),
  \begin{equation} \label{e:polchinski-generator}
    \ddp{}{t} \PP_{s,t}F = \LL_t  \PP_{s,t} F,
    \qquad
    -\ddp{}{s} \PP_{s,t}F = \PP_{s,t} \LL_s F,
    \qquad (s \leq t),
  \end{equation}
  for all smooth functions $F$, where $\LL_t$ acts on a smooth function $F$ by
  \begin{equation}  
    \label{e:polchinski-L}
    \LL_tF = \frac12 \Delta_{\dot C_t} F - (\nabla V_t, \nabla F)_{\dot C_t}.
  \end{equation}
  The measures $\nu_t$ evolve dual to $(\PP_{s,t})$ in the sense that
  \begin{equation} \label{e:polchinski-semigroup}
    \E_{\nu_t}\PP_{s,t} F = \E_{\nu_s} F \quad (s \leq t),
        \qquad
    -\ddp{}{t} \E_{\nu_t} F = \E_{\nu_t} \LL_t F.
  \end{equation}
  Finally, for any smooth function $F$ with values in a compact subset of $(0,\infty)$ and $\Phi(x) = x\log x$,
  \begin{equation} \label{e:nusQPs}
    \qquad \E_{\nu_t}\Phi(\PP_{0,t}F)
    \quad \text{is continuous in $t\in [0,+\infty].$}
  \end{equation}
\end{proposition}

\begin{remark}
  The Polchinski semigroup operates from the right, i.e., $\PP_{s,t} = \PP_{r,t}\PP_{s,r}$ for $s\leq r \leq t$.
  Thus it acts on probability densities relative to $\nu_t$:
  if $\mu_0 = F \, d\nu_0$ is a probability measure then $\mu_t = \PP_{0,t}F \, d\nu_t$ is again a probability measure.
  For a time-independent semigroup $\TT_{s,t} = \TT_{t-s}$ that is reversible with respect to the measure $\nu_0$
  (as, for example, the original semigroup associated to the Dirichlet form),
  one has the dual point of view
  that $\TT$ describes the evolution of an observable:
  \begin{equation}
    \E_{\mu_t} G
    =
    \int G (\TT_tF) \, d\nu_0 = \int (\TT_tG) F\, d\nu_0
    = \E_{\mu_0} (\TT_tG).
  \end{equation}
  Such a dual semigroup can be realised in terms of a Markov process $(\varphi_t)$ as
  $\TT_{t}F(\varphi) = \EE_{\varphi_0=\varphi}F(\varphi_t)$.
  Since the Polchinski semigroup is not reversible and time-dependent, this interpretation
  does not apply to the Polchinski semigroup. 
  Instead, the Polchinski semigroup $\PP_{s,t}$ can be realised in terms of an SDE that starts at time $t$
  and runs time in the negative direction from $t$ to $s$.
  Indeed, set $\varphi_{r} = \tilde\varphi_{t-r}$ where $\tilde\varphi$ satisfies
  \begin{equation}
    \label{e:stochastic}
    d \tilde\varphi_r = - \dot C_{t-r}  \, \nabla V_{t-r} (\tilde\varphi_r) dr+ \sqrt{\dot C_{t-r}} d B_r ,\qquad 0 \leq r \leq t.
  \end{equation}
  Since $G(r,\varphi) = \PP_{s,t-r}F(\varphi)$ satisfies $\partial_r G + \LL_{t-r}G =0$ for $s<r<t$ by \eqref{e:polchinski-generator},
  It\^o's formula and \eqref{e:stochastic} imply that $G(r,\tilde\varphi_r) = \PP_{s,t-r}F(\varphi_{t-r})$ is a martingale for $r\in [s,t]$.
  This implies
  \begin{equation} \label{e:stochastic2}
    \PP_{s,t}F(\varphi)= \E_{\varphi_t=\varphi} F(\varphi_s).
  \end{equation}
  Thus if $\varphi_t$ is distributed according to $\nu_t$ by the above backward in time evolution $\varphi_s$ is distributed
  according to $\nu_s$ for $s<t$.
  Our interpretation of this is that,
  while the renormalised measures $\nu_t$ are supported on increasing smooth (in the index space) configurations as $t$ grows,
  the backward evolution restores the small scale fluctuations of $\nu_0$.
\end{remark}

For later use we also record the following useful relations for the derivatives of $V_t$;
we will not use these in Section~\ref{sec:LSI}.
The formulas follow immediately by differentiating \eqref{e:V-def-bis} using \eqref{e:P-def-bis}.

\begin{proposition}
  For all $f \in X$ and $t\geq s \geq 0$,
  \begin{align}
    \label{e:PtdV}
    (f,\nabla V_t) &= \PP_{s,t}(f,\nabla V_{s}),\\
    \label{e:PtHessV}
    (f,\He V_tf) &= \PP_{s,t} (f, \He V_{s}f)
                   - \qB{ \PP_{s,t}((f, \nabla V_{s})^2) -(\PP_{s,t}(f, \nabla V_{s}))^2}.
  \end{align}
\end{proposition}

\subsection{Relative entropy, Markov semigroups, and the Bakry--\'Emery argument}

In a time-dependent generalisation,
we now review the decomposition of the relative entropy in terms
of a semigroup that underlies the Bakry--\'Emery argument.
By approximation (see, e.g., \cite[Theorem~3.1.13]{MR2352327}),
to prove a Log-Sobolev inequality,
it suffices to consider smooth functions $F: X\to \R$ with values in a compact subset of $(0,\infty)$,
which we will do from now on.

We consider a curve of probability measures $(\nu_t)_{t\geq 0}$
and a corresponding dual time-dependent Markov semigroup $(\PP_{s,t})$
with generators $(\LL_t)$ as in Proposition~\ref{prop:polchinski}.
Namely, we assume that \eqref{e:polchinski-generator} and \eqref{e:polchinski-semigroup} hold,
that $\LL_t$ is of the form \eqref{e:polchinski-L} for some positive semidefinite matrices $\dot C_t$
and functions $V_t$ (not necessarily satisfying \eqref{e:polchinski-bis}),
and also that \eqref{e:nusQPs} holds.
Denoting $F_t = \PP_{0,t}F$ and $\dot F_t=\ddp{}{t} F_t$,
using first \eqref{e:polchinski-semigroup} and then \eqref{e:polchinski-L},
it is then elementary to see that
\begin{align} \label{e:dEnt}
  - \ddp{}{t} \E_{\nu_t} \Phi (F_{t})
  &=
    \E_{\nu_t} 
    \pbb{
    \LL_t (\Phi (F_{t}))
    -
    \Phi '(F_{t})\dot{F}_{t}
    }
  \nnb
  &
    =
    \E_{\nu_t}
    \pbb{
    \Phi' (F_{t})\LL_t F_t
    +
    \Phi''(F_{t}) \frac{1}{2} (\nabla F_{t})_{\dot C_t}^{2}
    -
    \Phi'( F_{t}) \dot{F}_{t}
    }
  \nnb
  &
    =
    \frac12
    \E_{\nu_t}
    \pbb{
    \Phi''(F_t) (\nabla F_{t})_{\dot C_t}^{2}
    }
    .
\end{align}
Integrating this relation using \eqref{e:nusQPs},
with $\Phi''(x)=1/x$,
it follows that
\begin{equation}
  \label{e:Ent-P}
  \ent_{\nu_0}(F)
  = \frac12 \int_0^\infty \E_{\nu_{t}} \frac{(\nabla \PP_{0,t} F)_{\dot C_t}^2}{\PP_{0,t}F} \, dt
  = 2\int_0^\infty \E_{\nu_{t}} (\nabla \sqrt{\PP_{0,t} F})_{\dot C_t}^2 \, dt
  .
\end{equation}
To be precise,
recall that $C_t$ is differentiable except for at most countably many $t$.
For all $t$ such that $C_t$ is differentiable, the identity \eqref{e:dEnt} holds
and implies that the continuous function $t\mapsto \E_{\nu_t} \Phi (F_{t})$
is differentiable at $t$ with nonpositive derivative.
In particular, this implies that $\E_{\nu_t} \Phi (F_{t})$ is decreasing,
which justifies the use of the fundamental theorem of calculus and
together with \eqref{e:continuity} with $t=+\infty$ for the limit gives \eqref{e:Ent-P}.
% Casey, Stephen & Holzsager, Richard. (2005). On positive derivatives and monotonicity. Missouri Journal of Mathematical Sciences. 17. 

To obtain a Log-Sobolev inequality, 
the right-hand side of \eqref{e:Ent-P} must be bounded
by the Dirichlet form with respect to the measure $\nu_0$.
The same argument with $\Phi(x)=x^2$ would give a bound on the variance rather than
the  entropy and correspondingly a spectral gap inequality;
the required bound is easier to obtain in this case.

For measures that are log-concave (or, more generally, ones that satisfy a curvature dimension condition; see \cite{MR3155209}),
sharp estimates have been obtained by celebrated arguments of
Lichnerowicz (for the spectral gap) and of Bakry--\'Emery.
We review the latter briefly now.

\begin{example}[Bakry--\'Emery \cite{MR889476,MR2213477}]
Assume the measure $\nu=\nu_0$ has expectation given by \eqref{e:measureH}.
Let $\nu_t= \nu_0$ for all $t \geq 0$, and define the semigroup $\TT_{s,t} = \TT_{t-s}$ with generator
\begin{equation} \label{e:BE-generator}
  \LL F
  = \Delta F - (\nabla H,\nabla F).
\end{equation}
This semigroup leaves $\nu_0$ invariant.
Bakry--\'Emery showed, for all $F\geq 0$,
\begin{align} \label{e:LSI-identity-BE}
  \ddp{}{t} \E_{\nu_0}(\nabla \sqrt{\TT_tF})^2 
  &= - \frac14 \E_{\nu_0}(\TT_tF (|\He \log \TT_tF|_2^2 + (\nabla \log \TT_tF, (\He H)\nabla \log \TT_tF)))
  \nnb
  &\leq -\frac14 \E_{\nu_0}(\TT_tF(\nabla \log \TT_tF, (\He H)\nabla \log \TT_tF)))
   .
\end{align}
If $\He H(\varphi) \geq \lambda \id > 0$ as quadratic forms, uniformly in $\varphi \in \R^N$, it follows that
\begin{align}
  \ddp{}{t} \E_{\nu_0}(\nabla \sqrt{\TT_tF})^2
  \leq
  - \lambda \E_{\nu_0}(\nabla \sqrt{\TT_tF})^2,
  \qquad
  \E_{\nu_0}(\nabla \sqrt{\TT_tF})^2 \leq e^{-\lambda t} \E_{\nu_0}(\nabla \sqrt{F})^2.
\end{align}
Substituting this into \eqref{e:Ent-P} yields the Log-Sobolev inequality
\begin{equation} \label{e:BE-LSI}
  \ent_{\nu_0}(F)
  = 2\int_0^\infty \E_{\nu_0}(\nabla \sqrt{\TT_tF})^2 \, dt
  \leq \frac{2}{\lambda} \E_{\nu_0}(\nabla \sqrt{F})^2.
\end{equation}
In fact, \eqref{e:LSI-identity-BE} follows as in Lemma~\ref{lem:LSI-pf-identity} below.
\end{example}

% \begin{remark}
%   Despite the formal similarities, the interpretation of the interpolation of the entropy by the canonical semigroup $(\TT_t)$
%   with generator \eqref{e:BE-generator} is quite different from that with the Polchinski semigroup $(\PP_{s,t})$.
%   Consider a probability measure $d\rho_0 = d\mu_0 = F\, d\nu_0$ and define its evolution by $\TT$ respectively $\PP$
%   by $d\rho_t = \TT_tF \, d\nu_0$ and $d\mu_t = \PP_{0,t}F \, d\nu_t$.
%   Write $H(\mu|\nu) = \ent_{\nu}(d\mu/d\nu)$ for the relative entropy of $\mu$ with respect to $\nu$.
%   In the case of the original
%   semigroup, the Bakry--\'Emery argument uses the ergodicity $\rho_t \to \nu_0$ in the sense $H(\rho_t|\nu_0) \to 0$.
%   On the other hand, in the case of the Polchinski semigroup, $\mu_t \to \delta_0$ becomes trivial, but also $\nu_t \to \delta_0$ becomes trivial
%   in such a way  that their relative entropy again vanishes asymptotically: $H(\mu_t|\nu_t) \to 0$ by \eqref{e:nusQPs}.
% \end{remark}

\subsection{Variations of Theorem~\ref{thm:LSI-heat}}

The following theorem generalises Theorem~\ref{thm:LSI-heat} by not assuming that $\dot C_t$
is given by the heat kernel.

\begin{theorem} \label{thm:LSI-mon}
  Let $\dot C_t$ and $V_t$ be as in Section~\ref{sec:LSI-setup},
  assume that $\dot C_t$ is differentiable for all $t$, and that \eqref{e:continuity} holds.
  Suppose there are $\dot\lambda_t$ (allowed to be negative) such that
  \begin{equation} \label{e:assCt-mon}
    \dot C_t \He V_t(\varphi) \dot C_t - \frac{1}{2} \ddot C_t \geq \dot\lambda_t \dot C_t
    \quad \text{for all $t \geq 0$ and all $\varphi \in X$},
  \end{equation}
  and define
  \begin{equation} \label{e:gammadef-mon}
    \lambda_t = \int_0^t \dot \lambda_s \, ds,
    \qquad \frac{1}{\gamma} = |\dot C_0| \int_0^\infty e^{-2\lambda_s} \, ds
  \end{equation}
  where $|\dot C_0|$ is the largest eigenvalue of $\dot C_0$.
  Then $\nu_0$ satisfies the Log-Sobolev inequality
  \begin{equation} \label{e:LSI-mon}
    \ent_{\nu_0}(F)
    \leq \frac{2}{\gamma} \E_{\nu_0}(\nabla \sqrt{F})^2.
  \end{equation}
\end{theorem}

The proof of the theorem is given in Section~\ref{sec:proof2}.
When $\dot C_t$ is given by the heat kernel,
as in the context of Theorem~\ref{thm:LSI-heat},
the term $\ddot C_t$ in \eqref{e:assCt-mon} can be eliminated
explicitly and we can thus deduce Theorem~\ref{thm:LSI-heat} as follows.

\begin{proof}[Proof of Theorem~\ref{thm:LSI-heat}]
  Let $Q_t = e^{-tA/2}$ and $\dot C_t = e^{-tA} = Q_t^2$. Then $\ddot C_t = -A\dot C_t = -Q_tA Q_t$ and
  the left-hand side of \eqref{e:assCt-mon} is equal to
  \begin{equation}
    Q_t \qa{Q_t \He V_t(\varphi)Q_t + \frac12 A} Q_t.
  \end{equation}
  Since by assumption $A \geq \lambda$ and $Q_t \He V_t Q_t \geq  \dot \mu_t$ we can choose 
  $\dot \lambda_t = \frac{1}{2}\lambda+ \dot \mu_t$ to get
  \begin{equation} 
    \frac{1}{2} A + Q_t \He V_t(\varphi)Q_t \geq \dot\lambda_t \id,
  \end{equation}
  which with $Q_t^2 = \dot C_t$ implies \eqref{e:assCt-mon}.
  The claim \eqref{e:gammadef-heat} is thus implied by Theorem~\ref{thm:LSI-mon}.
\end{proof}

The next theorem provides a variation of Theorem~\ref{thm:LSI-mon}
that does not rely on differentiability or even continuity of $\dot C_t$ in $t$,
and can therefore be applied with more general covariance decompositions.
The price is the less symmetric condition \eqref{e:assCt-asy}.
However, this condition can for example be applied to discrete decompositions
$C_\infty = C_0 + C_1 + \cdots$ by setting $\dot C_s = \sum_j 1_{(j,j+1]}(s) C_j$.
In particular, this applies to the hierarchical spin
models that we studied in \cite{MR4061408}; see Example~\ref{ex:hierarchical}.

\begin{theorem} \label{thm:LSI-asy}
  Let $\dot C_t$ and $V_t$ be as in Section~\ref{sec:LSI-setup},
  and let $X_t \subseteq X$ be the image of the matrix $C_\infty-C_t$.
  Assume that \eqref{e:continuity} holds and that there are $\dot\lambda_t$ (allowed to be negative) such that
  \begin{equation} \label{e:assCt-asy}
    \frac12 \qa{ \dot C_t \He V_t(\varphi) + \He V_t(\varphi) \dot C_t} \geq \dot\lambda_t \id
    \quad \text{for all $t \geq 0$ and all $\varphi\in X_t$},
  \end{equation}
  and define
  \begin{equation} \label{e:gammadef-asy}
    \lambda_t = \int_0^t \dot \lambda_s \, dt,
    \qquad \frac{1}{\gamma} = \int_0^\infty e^{-2\lambda_s} |\dot C_s| \, ds
  \end{equation}
  where $|\dot C_t|$ is the largest eigenvalue of $\dot C_t$.
  Then $\nu_0$ satisfies the Log-Sobolev inequality \eqref{e:LSI-mon}.
\end{theorem}

Again the proof is given in Section~\ref{sec:proof2}.

\begin{example}[Hierarchical models]
  \label{ex:hierarchical}
  Let $C_j = \mu_j Q_j$ be the decomposition of the hierarchical Green function
  as in \cite[Section~2.1]{MR4061408} (where we here write $\mu_j$ instead of $\lambda_j$)
  and set $\dot C_t = \sum_j 1_{(j,j+1]}(t) C_j$ and $\dot Q_t = \sum_j 1_{(j,j+1]}(t) Q_j$.
  Using the structure of the hierarchical decomposition,
  for $\varphi \in X_t$, the matrix $\He V_t(\varphi)$ is block diagonal
  with respect to scale-$j$ blocks (see \cite[Section~1.3]{MR4061408}) where $t \in (j,j+1]$ and constant on each such block.
  This means that $\He V_t(\varphi)$ commutes with $Q_t$ and by the hierarchical structure
  thus with $\dot C_t$. In particular, for $\varphi \in X_t$,
  \begin{equation} \label{e:hierbd}
    \dot C_t^{1/2} \He V_t(\varphi) \dot C_t^{1/2} \geq \dot\lambda_t \id
  \end{equation}
  implies \eqref{e:assCt-asy}.
  For hierarchical versions of the four-dimensional lattice $|\varphi|^4$ model
  in the approach of the critical point,
  and for the two-dimensional 
  lattice sine-Gordon model in the rough (Kosterlitz--Thouless) phase,
  we established the estimate \eqref{e:hierbd} for integer $t$ 
  (and appropriate $\dot\lambda_t$) in \cite{MR4061408}.
  By the same methods,
  one can extend those estimates to noninteger $t$
  with $-\dot\lambda_t = O(-\dot\lambda_j)$ for $t \in (j,j+1]$.
  Using Theorem~\ref{thm:LSI-asy} instead of  \cite[Theorem~2.1]{MR4061408},
  the theorems for the spectral gap in \cite{MR4061408} can thus
  be extended to analogous ones for the Log-Sobolev constant.
\end{example}

Further variations of the conditions \eqref{e:assCt-mon} and \eqref{e:assCt-asy} for the Log-Sobolev
inequality are possible and might be useful in other applications,
but we do not investigate these here.

\subsection{Proof of Proposition~\ref{prop:polchinski}}
\label{sec:proof1}

We start with the proof of Proposition~\ref{prop:polchinski}.
This is a straightforward computation from the definitions.

\begin{proof}[Proof of Proposition~\ref{prop:polchinski}]
Let $Z_t(\varphi)= \EE_{C_t} e^{-V_0(\varphi+\zeta)}$.
By a well-known computation (see, e.g., \cite[Section~2]{rg-brief}),
it follows that the Gaussian convolution
acts as the heat semigroup with time-dependent generator $\frac12 \Delta_{\dot C_t}$, i.e.,
if $Z_0$ is $C^2$ in $\varphi$ so is $Z_t$ for any $t>0$,
that $Z_t(\varphi)>0$ for any $t$ and $\varphi$, and that for any $t>0$
such that $C_t$ is differentiable, 
\begin{equation}
  \ddp{}{t} Z_t = \frac12 \Delta_{\dot C_t} Z_t, \quad Z_0=e^{-V_0}.
\end{equation}
Therefore $V_t = -\log Z_t$ satisfies the Polchinski equation
\begin{equation}
  \ddp{}{t} V_t = - \frac{\ddp{}{t} Z_t}{Z_t}
  = -\frac{\Delta_{\dot C_t} Z_t}{2Z_t}
  = -\frac12 e^{V_t} \Delta_{\dot C_t} e^{-V_t}
  = \frac12 \Delta_{\dot C_t} V_t - \frac12 (\nabla V_t)_{\dot C_t}^2
  .
\end{equation}

That $(\PP_{s,t})$ is a semigroup, i.e., that $\PP_{r,t}\PP_{s,r} = \PP_{s,t}$
and $\PP_{t,t} = \id$ for any $s\leq r\leq t$,
follows immediately from the definition \eqref{e:P-def}
and the convolution property of Gaussian measures, i.e.,
that the sum of two independent Gaussian vectors is Gaussian with covariance
given by the sum of the covariances (again see, e.g., \cite[Section~2]{rg-brief}).
The Markov property is obvious.
To verify that its generator $\LL_t$ is given by \eqref{e:polchinski-L},
set $F_t(\varphi) = \PP_{0,t}F(\varphi) = e^{V_t(\varphi)} \EE_{C_t}(e^{-V_0(\varphi+\zeta)} F(\varphi+\zeta))$.
Then
\begin{align}
  \ddp{}{t} F_t
  &= (\ddp{}{t} V_t) F_t + e^{V_t} \frac12 \Delta_{\dot C_t} \E_{C_t}(e^{-V_0(\cdot+\zeta)} F(\cdot+\zeta))
    \nnb
  &= (\ddp{}{t} V_t) F_t + e^{V_t} \frac12 \Delta_{\dot C_t} (e^{-V_t} F_t)
    \nnb
  &= (\ddp{}{t} V_t) F_t - (\frac12 \Delta_{\dot C_t} V_t) F_t + \frac12 (\nabla V_t)_{\dot C_t}^2 F_t + \frac12 \Delta_{\dot C_t} F_t - (\nabla V_t, \nabla F_t)_{\dot C_t}
    \nnb
  &= \frac12 \Delta_{\dot C_t} F_t - (\nabla V_t, \nabla F_t)_{\dot C_t}
    \nnb
  &= \LL_t F_t
    ,
\end{align}
which is the second equality in \eqref{e:polchinski-generator}.
The third inequality in \eqref{e:polchinski-generator} follows analogously,
and the first inequality is clear from the fact that the Gaussian measure
with covariance $0$ is the Dirac measure at $0$.

The first equality in \eqref{e:polchinski-semigroup} holds by definition,
and the second one is a direct computation  from the definition \eqref{e:V-def} and the fact that $V$ satisfies \eqref{e:polchinski}:
\begin{align}
  -\ddp{}{t} \E_{\nu_t} F
  &= \E_{\nu_t}((\ddp{}{t} V_t)F  - \frac12 (\Delta_{\dot C_t} V_t)F + \frac12 (\nabla V_t)_{\dot C_t}^2F + \frac12 \Delta_{\dot C_t}F - (\nabla V_t, \nabla F)_{\dot C_t} )
  \nnb
  &= \E_{\nu_t}(\frac12 \Delta_{\dot C_t}F - (\nabla V_t, \nabla F)_{\dot C_t} )
  = \E_{\nu_t}\LL_t F.
\end{align}

Finally, \eqref{e:nusQPs} follows from \eqref{e:continuity}.
Indeed, if $F$ takes values in a compact interval $I \subset (0,\infty)$,
then $\PP_{0,t}F$ also takes values in $I$. The function $\Phi$ is smooth on $I$
and can be extended to a bounded smooth function $g$ on $\R$
such that $g|_I = \Phi|_I$. The claim now follows from \eqref{e:continuity}.
\end{proof}

\subsection{Proofs of Theorems~\ref{thm:LSI-mon}-\ref{thm:LSI-asy}}
\label{sec:proof2}

Theorems~\ref{thm:LSI-mon}-\ref{thm:LSI-asy} can be proved in the same way as the
Bakry--\'Emery criterion with the crucial difference that the original semigroup is replaced by the Polchinski semigroup,
that the corresponding potentials depend on time, and that gradients are taken in terms of a time-dependent quadratic form.
We present the primary proofs along the lines of \cite{MR3155209};
see Remark~\ref{rk:coupling} for alternative proofs using synchronous coupling as in \cite{MR3330820}.

\begin{lemma}\label{lem:LSI-pf-identity}
  Let $\LL_t$, $\PP_{0,t}$, $\dot C_t$, $V_t$ be as in Section~\ref{sec:LSI-setup}.
  Then the following identity holds
  for any $t$-independent positive definite matrix $Q$:
  \begin{equation} \label{e:LSI-pf-identity}
    (\LL_{t}-\partial_t)(\nabla \sqrt{\PP_{0,t}F})^2_{Q}
    = 2(\nabla \sqrt{\PP_{0,t}F},\He V_t \dot C_t \nabla \sqrt{\PP_{0,t}F})_{Q}
    + \frac14 (\PP_{0,t}F) |\dot C_t^{1/2} (\He \log \PP_{0,t}F) Q^{1/2}|_2^2
    ,
  \end{equation}
  where $|M|_2^2 = \sum_{p,q}|M_{pq}|^2$ denotes the squared Frobenius norm of a matrix $M=(M_{pq})$.
\end{lemma}

\begin{proof}
  Throughout the proof, we drop the fixed index $t$,
  i.e., write $F$ instead of $\PP_{0,t}F$, and $\LL$ for $\LL_t$,
  and similarly for $\dot C_t$ and $V_t$.
  Then the left-hand side of \eqref{e:LSI-pf-identity} can be written as
  \begin{equation} \label{e:3terms}
    \frac12 \qa{ \LL \frac{(\nabla F)_{Q}^2}{2F} - \frac{(\nabla \LL F,\nabla F)_{Q}}{F}
    + \frac{(\nabla F)_{Q}^2}{2F^2} \LL F}
    .
  \end{equation}
  To compute the three terms, 
  we denote derivatives by subscripts $i,j,k,l$, and use the summation convention for these subscripts.
  The first term then is
  \begin{equation}
    \LL \frac{(\nabla F)_Q^2}{2F}
    = \frac12 \dot C_{ij}Q_{kl}\qa{ (\frac{F_kF_l}{2F})_{ij} -  2V_{i} (\frac{F_kF_l}{2F})_j}
    = \frac12 \dot C_{ij}Q_{kl}\qa{ (\frac{F_{ik}F_l}{F} - \frac{F_kF_lF_i}{2F^2})_{j} -  2 V_{i} (\frac{F_kF_l}{2F})_j}
  \end{equation}
  where the last bracket can be expanded as
  \begin{equation}
    \qa{ \frac{F_{ijk}F_l + F_{ik}F_{jl}}{F} - \frac{F_{ik}F_lF_j}{F^2}- \frac{2F_{kj}F_lF_i + F_kF_lF_{ij}}{2F^2} + \frac{F_kF_lF_iF_j}{F^3} - 2 V_{i} (\frac{F_{jk}F_l}{F} - \frac{F_kF_lF_j}{2F^2})}
    .
  \end{equation}
  The sum of the second and third terms in \eqref{e:3terms} is
  \begin{multline}
    -\frac{(\nabla \LL F, \nabla F)_{Q}}{F}
    + \frac{(\nabla F)_{Q}^2}{2F^2} \LL F
    = \frac12 \dot C_{ij} Q_{kl} \qa{\frac{ -(F_{kij}- 2 V_{i}F_{kj}- 2 V_{ik}F_j) F_l}{F}
      + \frac{(F_{ij}-2 V_iF_j)F_kF_l}{2F^2}}
    \\
    = \frac12 \dot C_{ij} Q_{kl} \qa{2 V_{ik}\frac{F_jF_l}{F} -\frac{F_{kij}F_l}{F} +\frac{F_{ij}F_kF_l}{2F^2}
      + 2 V_i (\frac{F_{kj}F_l}{F} -\frac{F_jF_kF_l}{2F^2}) }.
  \end{multline}
  By adding all three terms, we obtain that \eqref{e:3terms} equals
  \begin{equation}
  \label{e:BE-LS}
    \frac12 \dot C_{ij} Q_{kl} \frac{V_{ik} F_jF_l}{F}
    +
      \frac14 \dot C_{ij}Q_{kl}\qa{\frac{F_{ik}F_{jl}}{F} - \frac{F_{ik} F_l F_j + F_{j l}F_i F_k }{F^2} + \frac{F_kF_lF_iF_j}{F^3}}
      .
  \end{equation}
  Using that for any given indices $i,j,k,l$,
  \begin{equation}
    (\log F)_{ik} = (\frac{F_i}{F})_k = \frac{F_{ik}}{F} - \frac{F_iF_k}{F^2},
    \qquad
    (\log F)_{jk} = (\frac{F_j}{F})_l = \frac{F_{jl}}{F} - \frac{F_jF_l}{F^2},
  \end{equation}
  equation \eqref{e:BE-LS} can be written as
  \begin{equation}
    \label{e:BE-LS-2}
    \frac12 \dot C_{ij}Q_{kl} \frac{V_{ki} F_jF_l}{F} + \frac14 F \dot C_{ij}Q_{kl} (\log F)_{ik}(\log F)_{jl}
    .
  \end{equation}
  Using that $2(\sqrt{F})_j = F_j/\sqrt{F}$ for the first term,
  and that, for any symmetric matrix $M$,
  \begin{align}
    \dot C_{ij}Q_{kl} M_{ik}M_{jl}
    =  \dot C_{ip}^{1/2}\dot C_{jp}^{1/2} Q^{1/2}_{kq} Q^{1/2}_{lq} M_{ik}M_{jl}
    &=  \dot C_{ip}^{1/2}\dot C_{jp}^{1/2} (MQ^{1/2})_{iq} (M Q^{1/2})_{jq}
      \nnb
    &= (\dot C^{1/2} M Q^{1/2})_{pq} (\dot C^{1/2} M Q^{1/2})_{pq}
  \end{align}
  for the second term,
  \eqref{e:BE-LS-2} can therefore be written as
  \begin{equation}
    2(\nabla \sqrt{F},\He V \dot C \nabla \sqrt{F})_Q
    + \frac14 F |\dot C^{1/2} (\He \log F) Q^{1/2}|_2^2
    .
    \qedhere
  \end{equation}
  \end{proof}

\begin{proof}[Proof of Theorem~\ref{thm:LSI-mon}]
  Lemma~\ref{lem:LSI-pf-identity} with $Q=\dot C_s$ implies
  \begin{multline} \label{e:LSI-pf-identity-t-mon}
    (\LL_{s}-\partial_s)(\nabla \sqrt{\PP_{0,s}F})^2_{\dot C_s}
    = 2(\nabla \sqrt{\PP_{0,s}F},\He V_s \dot C_s \nabla \sqrt{\PP_{0,s}F})_{\dot C_s}
    - (\nabla \sqrt{\PP_{0,s}F})^2_{\ddot C_s}
    \\
    + \frac14 (\PP_{0,s}F) |\dot C_s^{1/2} (\He \log \PP_{0,s}F) \dot C_s^{1/2}|_2^2
    .
  \end{multline}
  By the assumption \eqref{e:assCt-mon} and since the last term is positive, it follows that
  \begin{equation}
    (\LL_{s}-\partial_s)(\nabla \sqrt{\PP_{0,s}F})^2_{\dot C_s}
    \geq 2\dot\lambda_{s} (\nabla \sqrt{\PP_{0,s}F})^2_{\dot C_s}
    .
  \end{equation}
  Equivalently,
  $\psi(s) := e^{-2\lambda_t+2\lambda_s} \PP_{s,t} \left[ (\nabla \sqrt{\PP_{0,s}F})^2_{\dot C_s} \right]$ satisfies $\psi'(s) \leq 0$ for $s<t$.
  This implies
  \begin{equation}
    (\nabla \sqrt{\PP_{0,t}F})^2_{\dot C_t} = \psi(t) \leq \psi(0) 
    = e^{-2\lambda_t} \PP_{0,t} \left[ (\nabla \sqrt{F})^2_{\dot C_0} \right]
    \leq |\dot C_0|  \, e^{-2\lambda_t} \PP_{0,t} \left[ (\nabla \sqrt{F})^2 \right]
    .
  \end{equation}
  By \eqref{e:Ent-P}, then \eqref{e:LSI-mon} follows.
\end{proof}

\begin{proof}[Proof of Theorem~\ref{thm:LSI-asy}]
  Lemma~\ref{lem:LSI-pf-identity} with $Q=\id$ implies
  \begin{multline} \label{e:LSI-pf-identity-t-asy}
    (\LL_{s}-\partial_s)(\nabla \sqrt{\PP_{0,s}F})^2
    = 2(\nabla \sqrt{\PP_{0,s}F},\He V_s \dot C_s \nabla \sqrt{\PP_{0,s}F})
    \\
    + \frac14 (\PP_{0,s}F) |\dot C_s^{1/2} (\He \log \PP_{0,s}F)|_2^2
    .
  \end{multline}
  By the assumption \eqref{e:assCt-asy} and since the last term is positive, it follows that,
  on $X_s$,
  \begin{equation}
    (\LL_{s}-\partial_s)(\nabla \sqrt{\PP_{0,s}F})^2
    \geq 2\dot\lambda_{s} (\nabla \sqrt{\PP_{0,s}F})^2
    .
  \end{equation}
  Equivalently, pointwise on $X_t$,
  $\psi(s) := e^{-2\lambda_t+2\lambda_s} \PP_{s,t} \left[ (\nabla \sqrt{\PP_{0,s}F})^2 \right]$ satisfies $\psi'(s) \leq 0$ for $s<t$.
  This implies, on $X_t$,
  \begin{equation}
    (\nabla \sqrt{\PP_{0,t}F})^2_{\dot C_t}
    \leq
    |\dot C_t|(\nabla \sqrt{\PP_{0,t}F})^2
    = |\dot C_t| \psi(t) \leq |\dot C_t| \psi(0) 
    = |\dot C_t| e^{-2\lambda_t} \PP_{0,t} \left[ (\nabla \sqrt{F})^2 \right].
  \end{equation}
  Again by \eqref{e:Ent-P},
  using that $\nu_t$ is supported on $X_t$,
  \eqref{e:LSI-mon} follows.
\end{proof}

\begin{remark} \label{rk:coupling}
  Using the representation \eqref{e:stochastic}-\eqref{e:stochastic2}
  of the semigroup $\PP_{s,t}$ in terms of a stochastic process (that evolves backwards in time from $t$ to $s$),
  one can alternatively prove the theorems using synchronous coupling as in \cite{MR3330820}.
\end{remark}

\section{Application to the continuum sine-Gordon model}
\label{sec:SG}

In this section, we prove Theorems~\ref{thm:sg-4pi} and \ref{thm:sg-4pi-kawasaki} by applying Theorem~\ref{thm:LSI-heat}.
While it is not necessary, we find it clearest to rescale the continuum sine-Gordon model at scale $\epsilon$
to a unit lattice problem.

\subsection{Rescaling and heat kernel decomposition}

Identifying $\Omega_{\epsilon,L}$ with the unit lattice $\Lambda = \frac{1}{\epsilon}\Omega_{\epsilon,L}$,
the continuum sine-Gordon model $\nu_{\epsilon,L}$ is equivalent to a spin system whose coupling matrix is given by the nearest neighbour Laplacian
on $\Z^d$.
We will thus drop the subscripts $\epsilon,L$ now,
and write $\nu_0$ for the measure of the form \eqref{e:nu0-A} with $X=\R^\Lambda$ and
\begin{equation} \label{e:sg-AV0}
  A = -\Delta_\Lambda+\epsilon^2m^2,\qquad
  V_{0}(\varphi)
  = \sum_{x\in\Lambda} z\epsilon^{2-\beta/4\pi} \cos(\sqrt{\beta} \varphi_x),
\end{equation}
where $\Delta_\Lambda$ is the standard unit lattice Laplacian acting on the discrete torus of side length $L/\epsilon$.
We emphasise that throughout this section $\Delta_\Lambda$ denotes the lattice Laplacian on $\Lambda$ and not the
Laplacian on $\R^\Lambda$ which we denoted $\Delta_{\dot C_t}$ in the previous section.
Note that $\varphi$ is not rescaled.
As is natural in this normalisation, we normalise the Glauber Dirichlet form,
for $F: \R^\Lambda \to \R$, by
\begin{equation}
  \sum_{x\in\Lambda} \E_{\nu_0} \qa{ \pa{\ddp{F}{\varphi_x}}^2}.
\end{equation}
Note that in this normalisation
the Log-Sobolev constant of the non-interacting (Gaussian) model with $z=0$ scales as $\epsilon^{2}m^2$
(corresponding to the unit order Log-Sobolev constant $m^2>0$ in the continuum scaling).
Also note that the correlation length of the non-interacting model scales as $1/(m\epsilon)$,
making it natural to assume $L \geq 1/m$ as in the statements of the theorems.

In the following, we will use Theorem~\ref{thm:LSI-heat}
to prove the same scaling in $\epsilon$ for the Log-Sobolev constant of the interacting model.
To verify the assumptions of Theorem~\ref{thm:LSI-heat},
we will prove the following estimates on $V_t$ as defined in \eqref{e:V-def}.
We recall that $Q_t=e^{-tA/2}$ denotes the heat kernel on the index space $\Lambda$.

\begin{proposition} \label{prop:SG-HeVC}
Let $\beta<6\pi$, and $L>0$, $m>0$, and $z\in \R$. 
Then \eqref{e:ergodicity} holds, and for all $t \geq 0$,
\begin{equation} \label{e:SG-HeVC}
  Q_t\He V_t(\varphi)Q_t \geq \dot\mu_t \id,
\end{equation}
where $\mu_t = \int_0^t \dot\mu_s \, ds$ satisfies 
\begin{equation} \label{e:sg-mu-bd}
  |\mu_t|
  \leq \mu^*
\end{equation}
with $\mu^* = \mu^*(\beta,z,m,L)$ independent of $\epsilon>0$.
Moreover, there is $\delta_\beta>0$ such that if
\begin{equation} \label{e:small}
  Lm \geq 1, \quad \text{and} \quad |z|m^{-2+\beta/4\pi} \leq \delta_\beta,
\end{equation}
then the optimal bound satisfies $\mu^* = O_\beta(|z| m^{-2+\beta/4\pi})$ uniformly in $L$.
\end{proposition}

Indeed, Theorem~\ref{thm:sg-4pi} is an immediate consequence of these estimates and Theorem~\ref{thm:LSI-heat}.

\begin{proof}[Proof of Theorem~\ref{thm:sg-4pi}]
The smallest eigenvalue of $A$ is $\lambda = \epsilon^2m^2$.
By \eqref{e:gammadef-heat} and \eqref{e:sg-mu-bd}, therefore
\begin{equation}
  \frac{1}{\gamma} =
  \int_0^\infty e^{-\lambda t-2\mu_t} \, dt
  \leq
  e^{2\mu^*} \int_0^\infty e^{-\lambda t} \, dt
  =
  \frac{e^{2\mu^*}}{\lambda}
  =
  \frac{e^{2\mu^*}}{\epsilon^2m^2}
  ,
\end{equation}
and Theorem~\ref{thm:LSI-heat} implies that $\nu_0$ satisfies a Log-Sobolev inequality with constant $\gamma$.
In the continuum normalisation of the Dirichlet form
\eqref{e:Dirichlet-eps},
the sine-Gordon measure thus satisfies a Log-Sobolev inequality with constant given by
$m^2 e^{-2\mu^*}$. Moreover, if \eqref{e:small} holds,
then $m^2e^{-2\mu^*} = m^2 + O_\beta(m^{\beta/4\pi}|z|)$.
\end{proof}

The proof of Theorem~\ref{thm:sg-4pi-kawasaki} for Kawasaki dynamics is almost the same as that of Theorem~\ref{thm:sg-4pi}.
The constraint measure $\nu_0^0$ can be written as in \eqref{e:nu0-Cinfty}, with the degenerate
covariance matrix $C_\infty^0$ supported on the subspace $X = \R^\Lambda_0 = \{ \varphi \in \R^\Lambda: \sum_x \varphi_x=0 \}$
given by
\begin{equation} \label{e:Cinfty0}
  C_\infty^0 = PA^{-1}P, \qquad \text{where } 
 P\varphi_x = \varphi_x -  \frac{1}{|\Lambda|}\sum_{y \in\Lambda} \varphi_y.
\end{equation}
In unit lattice scaling, the Dirichlet form for Kawasaki dynamics is given, for $F: \R^\Lambda_0 \to \R$, by
\begin{equation}
  \sum_{x\sim y\in\Lambda} \E_{\nu^0_0} \qa{ \pa{\ddp{F}{\varphi_x}-\ddp{F}{\varphi_y}}^2}.
\end{equation}
We decompose the covariance matrix $C_\infty^0$ in terms of
\begin{equation}
  \dot C_t^0 = e^{-tA} P,
  \qquad
  Q_t^0 = e^{-tA/2} P,
\end{equation}
and define $V_t^0$ as in \eqref{e:V-def} with respect to $\dot C_t^0$.
From now on, we will refer to the case that $V_t$ is replaced by $V_t^0$ and $\dot C_t$
by $\dot C_t^0$ as the \emph{conservative case}.
Then the statement of
Proposition~\ref{prop:SG-HeVC} remains true in the conservative case.

\begin{proposition} \label{prop:SG-HeVC-kawasaki}
  Let $\beta<6\pi$, and $L>0$, $m>0$, and $z\in \R$.
  Then \eqref{e:ergodicity} holds, and for all $t \geq 0$,
  \begin{equation}
    Q_t^0 \He V_t^0(\varphi)Q_t^0 \geq \dot\mu_t P,
  \end{equation}
  where $\mu_t$ satisfies \eqref{e:sg-mu-bd} with the same bound on $\mu^*$
  if \eqref{e:small} holds.
\end{proposition}

Analogously as in the proof of Theorem~\ref{thm:sg-4pi},
we deduce Theorem~\ref{thm:sg-4pi-kawasaki} from Proposition~\ref{prop:SG-HeVC-kawasaki}.

\begin{proof}[Proof of Theorem~\ref{thm:sg-4pi-kawasaki}]
Since $\Lambda$ is a discrete torus of side length $L/\epsilon$,
the smallest nonzero eigenvalue of the lattice Laplacian $-\Delta_\Lambda$ on $\Lambda$ is of order $(\epsilon/L)^2$.
We thus denote the smallest nonzero eigenvalue of $-\Delta_\Lambda$ on $\Lambda$  by $\zeta^2\epsilon^2$.
Explicitly, as $\epsilon \to 0$,
\begin{equation}
  \zeta^2 \to (\frac{2\pi}{L})^2.
\end{equation}
As in the proof of Theorem~\ref{thm:sg-4pi},
with $\lambda$ the smallest eigenvalue on $X$ of $A=-\Delta_\Lambda + \epsilon^2m^2$,
\begin{equation}
  \frac{1}{\gamma^0} \leq \frac{e^{2\mu^*}}{\lambda}=  \frac{e^{2\mu^*}}{\epsilon^2(\zeta^2+m^2)},
\end{equation}
and Theorem~\ref{thm:LSI-heat} implies that $\nu_0^0$
satisfies a Log-Sobolev inequality with constant $\gamma^0$:
\begin{equation}
  \ent_{\nu_0^0}(F)
    \leq \frac{e^{2\mu^*}}{\epsilon^2(m^2+\zeta^2)} \E_{\nu_0^0}(\nabla F,P\nabla F)
    \leq \frac{e^{2\mu^*}}{\epsilon^4\zeta^2(m^2+\zeta^2)} \E_{\nu_0^0}(\nabla F, -\Delta_\Lambda P \nabla F)
\end{equation}
where the last inequality again uses that the smallest nonzero eigenvalue of the lattice Laplacian $-\Delta$ is $\epsilon^{2}\zeta^2$.
We emphasise that $\nabla$ denotes the continuous gradient on $\R^\Lambda$ while $\Delta_\Lambda$ is the lattice Laplacian on $\Lambda$.
Recalling the continuum normalisation of the Dirichlet form given by \eqref{e:Dirichlet-Kawa-eps},
and \eqref{e:sg-mu-bd},
this is the claim of Theorem~\ref{thm:sg-4pi-kawasaki}.
\end{proof}

\subsection{Outline, scaling conventions, and heat kernel}

To prove Propositions~\ref{prop:SG-HeVC}-\ref{prop:SG-HeVC-kawasaki}, we proceed in the following steps.
We first consider the main case \eqref{e:small}.
The proofs are simpler for $\beta<4\pi$ and we begin with this case in Section~\ref{sec:BK}.
In Sections~\ref{sec:6pi}-\ref{sec:6pi-ngt2}, we extend this analysis to the case $\beta<6\pi$.
Finally, in Section~\ref{sec:zlarge}, we show that a crude argument suffices to remove the assumption \eqref{e:small}
at the cost of constants that are uniform in $\epsilon$ but not in $L$. 

To prove Propositions~\ref{prop:SG-HeVC}-\ref{prop:SG-HeVC-kawasaki},
we will require estimates on the heat kernel decomposition
\begin{equation}
  C_t = \int_0^t \dot C_s \, ds, \qquad \dot C_s = Q_s^2 = e^{-s A}.
\end{equation}
In this section, we
set-up a convenient normalisation and also collect some elementary estimates.
We have chosen the heat kernel decomposition
(and not a finite range decomposition, for example)
to be able to directly apply Theorem~\ref{thm:LSI-heat}.
The \emph{characteristic length scale} of the heat kernel is defined by
\begin{equation} \label{e:elldef}
  \ell_t= (1\vee \sqrt{t}) \wedge \frac{1}{\epsilon m}
\end{equation}
and we set
\begin{equation}
  \lQ_t = \ell_t \uQ_t,
  \qquad
  \lCdot_t = \ell_t^2\uCdot_t,
  \qquad
  \vartheta_t = e^{-\frac12 m^2\epsilon^2 t}.
\end{equation}
Standard estimates on the heat kernel imply that
$\uCdot_t(x,y)$ is essentially supported on $|x-y| \lesssim \ell_t$
and the above normalisation is such that
$\lCdot_{\lambda^2 t}(\lambda x,\lambda y) \approx \lCdot_t(x,y)$ and $\lQ_t^2 =\lCdot_t$.
We will often express estimates in terms of these quantities
and in terms of $\ell_t$ (instead of $t$),
and write integrals over the scale in terms of the approximately scale invariant
measure $dt/\ell_t^2 \approx dt/t$ (instead of $dt$).
For estimates involving the heat kernels $\uQ_t,\uCdot_t,C_t$ and its scaled versions, we will always impose the following assumption:
\begin{equation} \label{e:assLm}
  Lm \geq 1, \quad \text{or} \quad t \leq \frac{1}{\epsilon^2}\pa{\frac{1}{m^2} \wedge L^2}.
\end{equation}

The next lemma provides some elementary estimates on the heat kernel.
These are sufficient for the case $\beta<4\pi$;
for $\beta>4\pi$ more precise estimates are required (and will be stated in the section
they are used).
All of these
estimates on the heat kernel are collected in Appendix~\ref{app:dotC}.

\begin{lemma} \label{lem:dotC}
  Assume \eqref{e:assLm}.
  For any $x \in \Lambda$,
  \begin{equation}
    \label{e:Cs-diag}
    C_t(x,x) =
    \frac{1}{2\pi} \log \ell_t + O(1),
    \qquad
    \sup_x \sum_y |\lCdot_t(x,y)|
    =
    O(\ell_t^2 \vartheta_t^2), 
  \end{equation}
  and the same estimates hold in the conservative case.
\end{lemma}

\begin{proof}
  This follows from standard estimates on the heat kernel on $\Z^2$,
  see Appendix~\ref{app:dotC}.
\end{proof}

Further we define the \emph{scale dependent coupling constant} $\lz_t$ and its microscopic version $z_t$ by
\begin{equation} \label{e:ztdef}
  \lz_t = \ell_t^2z_t,
  \qquad
  z_t = e^{-\frac{\beta}{2} C_t(0,0)} z_0,
  \qquad
  \text{where }
  z_0 = \epsilon^{2-\beta/4\pi} z.
\end{equation}
For later purposes, we will now collect some basic properties of this definition.
By \eqref{e:Cs-diag} and the definitions of $\lz_t$ and $\ell_t$, uniformly in $t>0$,
\begin{equation} \label{e:ztbd}
  \lz_t
  = O_\beta(|z|(\epsilon\ell_t)^{2-\beta/4\pi})
  = O_\beta(|z| m^{-2+\beta/4\pi})
  .
\end{equation}
In the following, we write $x \lesssim y$ or $x=O_\beta(y)$
if $|x| \leq C_\beta |y|$ for a $\beta$-dependent constant $C_\beta$.
For any $\beta<8\pi$, by \eqref{e:ztbd} then
\begin{equation} \label{e:intzs}
  \int_0^t |\lz_s| \vartheta_s^2 \, \lds
  \lesssim
  |\lz_t|,
\end{equation}
as is straightforward to check from the definitions.
For use in the proof for $\beta>4\pi$, we also record the following estimates (again straightforward from the definitions):
for all positive integers $n$,
\begin{alignat}{2}
  \label{e:intzlbd}
  \int_0^t |z_s|^n \ell_s^{2(n-1)} \vartheta_s^2 \, \lds
  &\lesssim \frac{1}{n} |z_t|^n (C_\beta \ell_t^2)^{n-1}
  &\qquad&\text{for $\beta<8\pi(1-1/n)$,} % (i.e., $n\beta/4\pi < 2(n-1)$),
  \\
  \label{e:intzlbd-improved}
  \int_0^t |z_s|^n \ell_s^{2(n-1)} \ell_s^{\beta/4\pi} \vartheta_s^2 \, \lds
  &\lesssim \frac{1}{n} |z_t|^n (C_\beta \ell_t^2)^{n-1} \ell_t^{\beta/4\pi}
  &\qquad&\text{for $\beta<8\pi$.}%  (i.e., $(n-1)\beta/4\pi < 2(n-1)$),
\end{alignat}

\subsection{Fourier representation}

To estimate the Hessian of the renormalised potential $V_t$,
we use the Brydges--Kennedy approach \cite{MR914427}.
Namely, for any function $V: \R^\Lambda \to \R$ that is $\frac{2\pi}{ \sqrt{\beta}}$-periodic in each variable,
we will write its Fourier series (assuming it converges absolutely) as
\begin{equation} \label{e:V-fourier}
  V(\varphi) = \sum_{n=0}^\infty V^{(n)}(\varphi),
  \qquad
  V^{(n)}(\varphi) = \frac{1}{n!} \sum_{\xi_1,\dots,\xi_n} \tilde V^{(n)}(\xi_1,\dots, \xi_n) e^{i \sqrt{\beta}\sum_{k=1}^n\varphi_{x_k}\sigma_k}
\end{equation}
where $\tilde V^{(n)}: (\Lambda \times \{\pm 1\})^n \to \R$ and
\begin{equation}
  \xi_i=(x_i,\sigma_i) \in \Lambda \times \{\pm 1\}.
\end{equation}
We think of $\xi_i$ as a particle with position $x_i$ and charge $\sigma_i$.
Since the index $n$ is determined from the number of arguments of $\tilde V^{(n)}$, we will often omit
it and write $\tilde V(\xi_1,\dots,\xi_n) = \tilde V^{(n)}(\xi_1,\dots,\xi_n)$.
The representation \eqref{e:V-fourier} is \emph{not} manifestly unique without further conditions,
but in the relevant cases we will in fact construct coefficients $\tilde V(\xi_1,\dots,\xi_n)$ such that \eqref{e:V-fourier} holds.

The initial potential $V_0$ of the sine-Gordon model corresponds to
\begin{equation} \label{e:V0}
  \tilde V_0(\emptyset) = 0,
  \qquad
  \tilde V_0(\xi_1) = z_0,
  \qquad
  \tilde V_0(\xi_1,\dots,\xi_n) = 0 \quad (n>1).
\end{equation}
Set
\begin{equation}
  \dot u_s(\xi_i,\xi_j) = \beta \dot C_s(x_i,x_j)\sigma_i\sigma_j,
  \qquad
  \ludot_s(\xi_i,\xi_j)
  = \ell_s^2 \dot u_s(\xi_i,\xi_j)
  = \beta \lCdot_s(x_i,x_j)\sigma_i\sigma_j
\end{equation}
and
\begin{equation} \label{e:Wdef}
  \dot W_s(\xi_1,\dots,\xi_n) = \frac12 \sum_{k,l \in [n]} \dot u_s(\xi_k,\xi_l),
\end{equation}
where $[n] = \{1,\dots,n\}$.
We define $u_s$ and $W_s$ analogously by replacing $\dot C_s$ by $C_s$.
For later use, we note that $W_t-W_s \geq 0$ holds for all arguments by positive definiteness of $\dot C_s$.

Then in terms of the Fourier representation \eqref{e:V-fourier}, the two terms on the right-hand side of the Polchinski equation
\eqref{e:polchinski} are represented by
\begin{align} \label{e:DeltaV-fourier}
  \frac12 \widetilde{ (\Delta_{\dot C_s} V)}(\xi_1,\dots,\xi_n)
  &= - \frac12 \sum_{i,j \in [n]} \dot u_s(\xi_i,\xi_j) \tilde V(\xi_1,\dots,\xi_n)
    \nnb
  &= - \dot W_s(\xi_1,\dots,\xi_n)  \tilde V(\xi_1,\dots,\xi_n)
    \intertext{and}
    \label{e:gradV2-fourier}
  \frac12 \widetilde  {(\nabla V,\nabla V)}_{\dot C_s} (\xi_1,\dots,\xi_n)
  &= - \frac12 \sum_{I_1 \dot\cup I_2= [n]} \tilde V(\xi_{I_1})\tilde V(\xi_{I_2})  \sum_{i\in I_1, j \in I_2} \dot u_s(\xi_i,\xi_j).
\end{align}
The sum over $I_1 \dot\cup I_2=[n]$ is over all nonempty disjoint subsets $I_1$ and $I_2$ of $[n]$ with $I_1\cup I_2=[n]$.
Moreover, given $\xi_1,\dots,\xi_n$ and $I = \{i_1,\dots,i_k\} \subset [n]$ we denote by
$\xi_I$ the vector $(\xi_{i_1},\dots,\xi_{i_k})$. 

Indeed, \eqref{e:DeltaV-fourier} is straightforward to verify
in the sense that if $V$ is given by \eqref{e:V-fourier} and $\widetilde{\Delta_{\dot C_s} V}$
by \eqref{e:DeltaV-fourier} then
\begin{equation}
  \Delta_{\dot C_s} V(\varphi) = \sum_n \frac{1}{n!} \sum_{\xi_1,\dots,\xi_n} \widetilde{ (\Delta_{\dot C_s} V)}(\xi_1,\dots,\xi_n)
  e^{i\sqrt{\beta}\sum_{k=1}^{n} \varphi_{x_k} \sigma_k}.
\end{equation}
To see \eqref{e:gradV2-fourier}, note that 
differentiating \eqref{e:V-fourier} gives
\begin{equation}
  \ddp{}{\varphi_x} V^{(p)}(\varphi) = \frac{1}{p!} \sum_{\xi_1,\dots,\xi_p} \tilde V(\xi_1,\dots,\xi_p) \sum_{k=1}^p i\sqrt{\beta}\sigma_k 1_{x=x_k} e^{i\sqrt{\beta}\sum_{k=1}^p \varphi_{x_k} \sigma_k}
\end{equation}
and thus
\begin{multline}
  (\nabla V^{(p)}, \nabla V^{(q)})_{\dot C_s}(\varphi)
  = \frac{-1}{p!q!} \sum_{\xi_1, \dots, \xi_{p+q}} \tilde V(\xi_1,\dots,\xi_p)\tilde V(\xi_{p+1},\dots,\xi_{p+q})\\
  \sum_{i=1}^p \sum_{j=p+1}^{p+q} \dot u_s(\xi_i,\xi_j)
  e^{i\sqrt{\beta}\sum_{k=1}^{p+q} \varphi_{x_k} \sigma_k}.
\end{multline}
Therefore taking the sum over $p$ and $q$,
using that the number partitions of $[n]$ into two subsets with $p$ and $q=n-p$ elements is $n!/(p!q!)$
and that $\tilde V$ is symmetric in its arguments, we find
\begin{equation}
  (\nabla V, \nabla V)_{\dot C_s}(\varphi)
  =
  \sum_n \frac{1}{n!} \sum_{\xi_1,\dots,\xi_n} \widetilde  {(\nabla V,\nabla V)}_{\dot C_s} (\xi_1,\dots,\xi_n)
  e^{i\sqrt{\beta}\sum_{k=1}^{n} \varphi_{x_k} \sigma_k}
\end{equation}
if $\widetilde {(\nabla V,\nabla V)}_{\dot C_s}$ is given by \eqref{e:gradV2-fourier}.

By \eqref{e:DeltaV-fourier}-\eqref{e:gradV2-fourier} and the Duhamel principle,
the Polchinski equation has the following formulation as an integral equation:
\begin{multline} \label{e:polchinski-fourier-duhamel}
  \tilde V_t(\xi_1,\dots,\xi_n) =
  e^{-W_t(\xi_1,\dots,\xi_n)} \tilde V_0(\xi_1,\dots,\xi_n)
  \\ + \frac12 \int_0^t ds \, \sum_{I_1 \dot\cup I_2= [n]}
  \sum_{i\in I_1,j\in I_2}\dot u_s(\xi_i,\xi_j)
  \tilde V_s(\xi_{I_1}) \tilde V_s(\xi_{I_2}) 
  e^{-(W_t(\xi_1,\dots,\xi_n)-W_s(\xi_1,\dots,\xi_n))}.
\end{multline}
For $n\leq 1$, the unique solution to \eqref{e:polchinski-fourier-duhamel} is simply
\begin{equation}
  \tilde V_t(\emptyset) = \tilde V_0(\emptyset) = 0,
  \qquad
  \label{e:V1}
  \tilde V_t(\xi_1) = e^{-\frac12 u_t(\xi_1,\xi_1)} \tilde V_0(\xi_1) = z_t,
\end{equation}
with $z_t$ defined in \eqref{e:ztdef}.
For $n>1$, $\tilde V_t(\xi_1,\dots,\xi_n)$ is then determined explicitly
by \eqref{e:polchinski-fourier-duhamel}
in terms of $\tilde V_s(\xi_1,\dots,\xi_k)$, $k<n$.
Hence by induction, \eqref{e:polchinski-fourier-duhamel}
has a unique solution for any $n$ and $t$.
This is summarised in the following lemma along with a uniqueness property.

\begin{lemma} \label{lem:polchinki-integral}
  The integral equation \eqref{e:polchinski-fourier-duhamel} has a unique
  solution $\tilde V$ for all $n$ and $t$.
  Moreover, if $V_t$ defined in terms of $\tilde V_t$ by \eqref{e:V-fourier} converges absolutely,
  locally uniformly in $t>0$, then $V_t$ is equal to \eqref{e:V-def},
  the convolution  solution of the Polchinski equation.
\end{lemma}

\begin{proof}
We have already shown that \eqref{e:polchinski-fourier-duhamel} has a unique solution.
For coefficients $\tilde V_t$ such that \eqref{e:V-fourier}
and its derivatives converge absolutely,
the function $V_t$ defined by \eqref{e:V-fourier} is smooth.
Moreover, for smooth $V_t$, the integral equation
\eqref{e:polchinski-fourier-duhamel} implies
the Polchinski equation \eqref{e:polchinski}.
Uniqueness of bounded solutions to the Polchinski equation by Remark~\ref{rk:uniqueness}
then implies that $V_t$
coincides with the convolution solution of the Polchinski equation.
\end{proof}

\subsection{Up to the first threshold: proof of Propositions~\ref{prop:SG-HeVC}-\ref{prop:SG-HeVC-kawasaki} for $\beta<4\pi$ assuming \eqref{e:small}}
\label{sec:BK}

The following proposition, due to \cite{MR914427},
gives good bounds when $\beta<4\pi$.
For completeness, we reproduce their argument here in our set-up and notation.
(See also \cite{MR630333,MR641914,MR833024,MR733471,MR4010781} for related results.)
We will then use the result to derive Proposition~\ref{prop:SG-HeVC}
in the case $\beta<4\pi$.
Let
\begin{equation}
  \|\dot u_s\| = \sup_{\xi_1} \sum_{\xi_2} |\dot u_s(\xi_1,\xi_2)|
\end{equation}
and
\begin{equation}
  \nV{1} = \sup_{\xi_1} |\tilde V(\xi_1)|,
  \qquad
  \nV{n} = \sup_{\xi_1} \sum_{\xi_2,\dots,\xi_n} |\tilde V(\xi_1,\dots,\xi_n)| \quad (n>1).
\end{equation}

\begin{proposition} \label{prop:Vtbd}
For all $n \geq 1$, the solution to \eqref{e:polchinski-fourier-duhamel} satisfies
\begin{equation} \label{e:Vtbd}
  \nVt{n}
  \leq n^{n-2} |z_t|^n M_t^{n-1},
  \quad
  \text{where}
  \quad
  M_t = \int_0^t ds \| \dot u_s \| e^{\beta (C_t-C_s)(0,0)}
  ,
\end{equation}
with $z_t$ defined in \eqref{e:ztdef}.
In particular, if $z_tM_t < 1/e$, the Fourier series for $V_t$ converges
and $V_t$ coincides with the convolution solution to the Polchinski equation.
The analogous statements hold in the conservative case.
\end{proposition}

\begin{proof}
For $n=1$, the bound \eqref{e:Vtbd} is obvious from \eqref{e:V1}.
To prove the bounds \eqref{e:Vtbd} for $n>1$, we use induction.
Note that the first term on the right-hand side of \eqref{e:polchinski-fourier-duhamel}
does not contribute for $n>1$ since then $\tilde V_0^{(n)} = 0$ by \eqref{e:V0}.
In the second term, we drop the exponential inside the integral 
(as $W_t - W_s \geq 0$) to obtain
\begin{equation} \label{e:polchinski-fourier-duhamel-bis}
  |\tilde V_t(\xi_1,\dots,\xi_n)| \leq
  \frac12 \int_0^t ds \, \sum_{I_1 \dot\cup I_2= [n]}
  \sum_{i\in I_1,j\in I_2} |\dot u_s(\xi_i,\xi_j)
  \tilde V_s(\xi_{I_1}) \, \tilde V_s(\xi_{I_2})|
  .
\end{equation}
Note that if $|I_1|=n-k$ and $|I_2|=k$ then
\begin{align}
  \sup_{\xi_1} \sum_{\xi_2,\dots,\xi_n}  |\dot u_s(\xi_i,\xi_j)
  \tilde V_s(\xi_{I_1}) \tilde V_s(\xi_{I_2})| \leq \|\dot u_s\| \nVs{n-k} \nVs{k} 
  .
\end{align}
For example,
\begin{align}
  &\sup_{\xi_1} \sum_{\xi_2,\xi_3,\xi_4} |\dot u_s(\xi_1,\xi_3) \tilde V_s(\xi_1,\xi_2) \tilde V_s(\xi_3,\xi_4)|
  \nnb
  &\leq \sup_{\xi_1} \sum_{\xi_3} |\dot u_s(\xi_1,\xi_3)| \sup_{\xi_1} \sum_{\xi_2} |\tilde V_s(\xi_1,\xi_2)| \sup_{\xi_3} \sum_{\xi_4} |\tilde V_s(\xi_3,\xi_4)|
    \leq \|\dot u_s\| \nVs{2}^2 
    .
\end{align}
Assuming the bound \eqref{e:Vtbd} for integers less than $n$, therefore
\begin{align}
  \nVt{n}
  &\leq \frac12 \int_0^t ds \,  \|\dot u_s\| \sum_{k=1}^{n-1} \binom{n}{k} k(n-k) 
    \nVs{n-k}\,\nVs{k} 
  \nnb
  &\leq
    \frac12 \int_0^t ds \, \|\dot u_s\| \sum_{k=1}^{n-1} \binom{n}{k}
    |z_s|^n M_s^{n-2} (n-k)^{n-k-1}k^{k-1}
    .
\end{align}
Using that $\sum_{k=1}^{n-1} \binom{n}{k} k^{k-1} (n-k)^{n-k-1} = 2(n-1)n^{n-2}$ and $n/2 \leq n-1$ for $n\geq 2$,
\begin{align}
  \nVt{n} 
  &\leq
    n^{n-2} |z_t|^n (n-1) \int_0^t ds \, \|\dot u_s\| e^{\frac{n}{2} \beta (C_t-C_s)(0,0)} M_s^{n-2} 
  \nnb
  &\leq
    n^{n-2} |z_t|^n  (n-1) \int_0^t ds \, \|\dot u_s\| e^{(n-1)\beta (C_t-C_s)(0,0)} M_s^{n-2} 
  =
    n^{n-2} |z_t|^n M_t^{n-1}
    .
\end{align}
For $n>2$, the last equality follows from the following change of variables,
\begin{equation}
  (n-1) \int_0^t ds \, g(s) \pa{\int_0^s ds' \, g(s')}^{n-2}
  = \pa{\int_0^t ds \, g(s)}^{n-1}
  ,
\end{equation}
applied with $g(s)=\|\dot u_s\|e^{-\beta C_s(0,0)}$.
Indeed,
\begin{multline}
  (n-1) \int_0^t ds \, \|\dot u_s\|  e^{\beta (n-1)(C_t-C_s)(0,0)}M_s^{n-2} 
  \\=
    (n-1) e^{\beta (n-1)C_t(0,0)} \int_0^t ds \, \|\dot u_s\|  e^{-\beta C_s(0,0)} \pa{\int_0^s ds' \, \|\dot u_{s'}\| e^{-\beta C_{s'}(0,0)}}^{n-2} 
    = M_t^{n-1}
    .
\end{multline}

Finally, using the bounds \eqref{e:Vtbd} for 
$\tilde V_t(\xi_1,\dots,\xi_n)$ and the assumption $\sup_t z_tM_t < 1/e$, 
the series \eqref{e:V-fourier} for $V_t(\varphi)$ converges absolutely since
(using $n^n/n! \leq e^n$),
\begin{equation} \label{e:Vconverges}
  \frac{|V_t(\varphi)|}{|\Lambda|}
  \leq
  \sum_{n=1}^\infty \frac{1}{n!} n^{n-2} |z_t|^n M_t^{n-1}
  \leq \sum_{n=1}^\infty e^n |z_t|^n M_t^{n-1}
  = \frac{e|z_t|}{1-e|z_t|M_t}
  \leq C <\infty,
\end{equation}
and analogously for derivatives.
Hence $V$ solves the Polchinski equation \eqref{e:polchinski}
by Lemma~\ref{lem:polchinki-integral}.
\end{proof}

Using the conclusion of the last proposition
together with the basic estimates for $\lCdot_s$ given
in Lemma~\ref{lem:dotC}, it is straightforward to
complete the proof of Propositions~\ref{prop:SG-HeVC}-\ref{prop:SG-HeVC-kawasaki} for $\beta<4\pi$.

\begin{proof}[Proof of Propositions~\ref{prop:SG-HeVC}-\ref{prop:SG-HeVC-kawasaki} for $\beta<4\pi$ assuming \eqref{e:small}]
Since the proofs of the two propositions are identical we only discuss Proposition~\ref{prop:SG-HeVC}.
From \eqref{e:Cs-diag},
\begin{equation}
 \| \dot u_s \| \leq \beta  \vartheta_s^2 \,\sup_x \sum_y |\dot C_s(x,y)|  \leq O_\beta( \vartheta_s^2).
\end{equation}
For $\beta<4\pi$, the definition of $M_t$ in \eqref{e:Vtbd}, 
the definition of $\ell_t$ in \eqref{e:elldef},
and \eqref{e:Cs-diag} imply
\begin{equation} \label{e:Mtbd}
  M_t \leq C_\beta \ell_t^{\beta/(2\pi)}  \int_0^t ds \, \vartheta_s^2 \, \ell_s^{- \beta/(2\pi)}
  =O_\beta(\ell_t^2).
\end{equation}
In this proof, the condition $\beta<4\pi$ is only needed in order to achieve the scaling $\ell_t^2$ in the previous upper bound.
By \eqref{e:ztdef}-\eqref{e:ztbd} therefore, using in the last inequality that $|z|   m^{-2+\beta/4\pi}$ is sufficiently small,
\begin{equation}
  |z_t| M_t
  = O_\beta(|\lz_t|)
  = O_\beta(|z| m^{-2+\beta/4\pi})
  \leq \frac{1}{2e}
  .
\end{equation}
Let
\begin{equation}
  \|\He V_t(\varphi)\| =
  \sup_x \sum_y |\ddp{^2}{\varphi_x\partial\varphi_y} V_t(\varphi)|.
\end{equation}
From \eqref{e:V-fourier} together with \eqref{e:Vtbd}, \eqref{e:Mtbd}, and with $n^n/n! \leq e^{n}$ we obtain
\begin{equation} \label{e:HeVsformbd}
  \|\He V_t(\varphi)\|
  \leq
\beta  \sum_{n=1}^\infty \frac{1}{n!}   n^2 n^{n-2} |z_t|^n M_t^{n-1}
  \leq   \beta \sum_{n=1}^\infty e^n |z_t|^n M_t^{n-1}
  = \frac{   \beta e|z_t|}{1-e|z_t|M_t}
  \leq 2    \beta e |z_t|.
\end{equation}
Since $|(f,\He V_t(\varphi)f)| \leq \|\He V_t(\varphi)\||f|_2^2$
and $|Q_tf|_2 \leq \vartheta_t |f|_2$, we obtain
\begin{equation} \label{e:HessCbd}
  |(Q_tf,\He V_t(\varphi) Q_tf)| \leq O_\beta(|z_t| \vartheta_t^2) |f|_2^2. 
\end{equation}
In the notation of Theorem~\ref{thm:LSI-heat} we thus have that $\dot\mu_t \geq - O_\beta(|z_t| \vartheta_t^2)$.
Hence, using the bounds for $z_t$ from \eqref{e:intzs} and \eqref{e:ztbd}, for all $t\geq 0$,
\begin{equation}
  \mu_t
  \geq - \int_0^t O_\beta(|\lz_s|\vartheta_s^2) \, \lds
  \geq - O_\beta(|\lz_t|)
  \geq - O_\beta(|z| m^{-2+\beta/4\pi})
  \equiv - \mu^*.
\end{equation}

Finally, the ergodicity assumption \eqref{e:ergodicity} follows from
the weak-* convergence $\nu_t\to\nu_\infty \equiv \delta_0$ and
$\PP_{0,t}F(\varphi) \to \PP_{0,\infty} F(\varphi)$ uniformly in $\varphi$.
Indeed, $\nu_t\to\nu_\infty$ holds since
the Gaussian measure covariance $C_\infty-C_t$ converges to $\delta_0$
and $V_t(\varphi)$ is bounded (uniformly in $\varphi$ and $t$).
The uniform convergence $\PP_{0,t}F \to \PP_{0,\infty} F$ holds
since $V_t(\varphi) \to V_\infty(\varphi)$ and
$\EE_{C_s} e^{-V_0(\varphi+\zeta)} F(\varphi+\zeta) \to \EE_{C_\infty} e^{-V_0(\varphi+\zeta)} F(\varphi+\zeta)$,
both uniformly in $\varphi$,
where the last claim holds since the integrand is a bounded Lipschitz function.
\end{proof}

\subsection{Up to the second threshold:
  proof of Propositions~\ref{prop:SG-HeVC}-\ref{prop:SG-HeVC-kawasaki} for $\beta<6\pi$ assuming \eqref{e:small}}
\label{sec:6pi}

The remainder of Section~\ref{sec:SG} is devoted to extending the proof of
Proposition~\ref{prop:SG-HeVC} from $\beta<4\pi$ to $\beta<6\pi$.
For this, we will estimate the $n=2,3,4$ terms in \eqref{e:V-fourier} more carefully.

Indeed, for $n=2$,
a uniform bound on $\tilde V_t(\xi_1,\xi_2)$ as used for $\beta<4\pi$ is not true when $\beta \geq 4\pi$,
and we rely crucially on the smoothing effect of the heat kernel $Q_t$ in \eqref{e:assVt-heat} to obtain
the required bound stated in the following proposition.
(Note that this estimate is best expressed in terms of $\lQ_t$ and $\lz_t$ rather than $Q_t$ and $z_t$.)

\begin{proposition} \label{prop:Vtbd-thres1-2}
  Let $\beta<8\pi$ and assume \eqref{e:assLm}. Then
  \begin{equation} \label{e:V-2-bd}
    (\lQ_t f,\He V_t^{(2)}(\varphi) \lQ_t f)
    = O_\beta(|\lz_t|^2 \vartheta_t^2)|f|_2^2.
  \end{equation}
  The analogous statement holds in the conservative case.
\end{proposition}

For the terms $n>2$, the following proposition gives an analogue of Proposition~\ref{prop:Vtbd}
for $\beta<6\pi$.

\begin{proposition} \label{prop:Vtbd-thres1-no2}
  Let $\beta < 6\pi$ and assume \eqref{e:assLm}.
  Then there is $C_\beta<\infty$ such that for all $n\geq 3$,
  \begin{align} \label{e:Vtbd-thresh1}
    \nVt{n} \leq n^{n-2} |z_t|^{n} (C_\beta \ell_t^2)^{n-1}.
  \end{align}
  The analogous statement holds in the conservative case.
\end{proposition}

These bounds together imply Propositions~\ref{prop:SG-HeVC}-\ref{prop:SG-HeVC-kawasaki} when \eqref{e:small} holds.

\begin{proof}[Proof of Propositions~\ref{prop:SG-HeVC}-\ref{prop:SG-HeVC-kawasaki} assuming \eqref{e:small}]
Since the proofs are again the same, and we only prove Propositions~\ref{prop:SG-HeVC}.
The bound \eqref{e:Vtbd-thresh1} (together with the qualitative fact that $V^{(1)}$ and $V^{(2)}$ are finite)
implies that \eqref{e:V-fourier} converges, exactly as in \eqref{e:Vconverges}.
Moreover, exactly as in \eqref{e:HeVsformbd}-\eqref{e:HessCbd},
for $|z| m^{-2+\beta/4\pi}$ sufficiently small, it follows that
\begin{equation} \label{e:Vtbd-thresh1-He}
  (\lQ_tf,(\He V_t(\varphi)-\He V_t^{(2)}(\varphi))\lQ_tf) = O_\beta(|\lz_t| \vartheta_t^2)|f|_2^2.
\end{equation}
Combined with \eqref{e:V-2-bd} this gives the required bound  \eqref{e:SG-HeVC}.
The proof of the ergodicity assumption \eqref{e:ergodicity} is also identical
to that in the proof of Proposition~\ref{prop:SG-HeVC} for $\beta<4\pi$.
\end{proof}

To prove the above propositions, \emph{neutral} configurations require more careful
treatment compared to the case $\beta<4\pi$, where neutral means the following.
For a configuration $\xi = (\xi_1,\dots,\xi_k)$ we define the \emph{charge}
$\sigma(\xi)=\sum_{i=1}^k \sigma_i$ and
call $\xi$ \emph{neutral} if $\sigma(\xi)=0$ and call $\xi$ \emph{charged} otherwise.
We will sometimes decompose
\begin{gather}
  V^{(n)}(\varphi) = V^{(n,0)}(\varphi)+V^{(n,\pm)}(\varphi)
  \\
  \tilde V^{(0)}(\xi)
  =   \tilde V(\xi)1_{\sigma(\xi)=0}
  ,\quad
  \tilde V^{(\pm)}(\xi)
  =   \tilde V(\xi)1_{\sigma(\xi)\neq 0},
\end{gather}
where $V^{(n,0)}$ is defined as in \eqref{e:V-fourier} with the sum over
$\xi=(\xi_1,\dots,\xi_n)$ restricted to neutral $\xi$,
and $V^{(n,\pm)}$ by restricting the sum to charged $\xi$.
As in the proof for $\beta<4\pi$,
the starting point for the proofs is
\eqref{e:polchinski-fourier-duhamel},
but now without dropping the exponential inside the integral,
i.e., for $n > 1$,
\begin{align} \label{e:polchinski-fourier-duhamel-bis3}
  \tilde V_t(\xi_1,\dots,\xi_n)
  &=
  -\frac12
  \sum_{I_1 \dot\cup I_2= [n]}
  \int_0^t ds \, 
    \qbb{
      \sum_{i\in I_1,j\in I_2} \dot u_s(\xi_i,\xi_j)
      \tilde V_s(\xi_{I_1})\tilde V_s(\xi_{I_2})
    } e^{-(W_t(\xi)-W_s(\xi))}
    \nnb
    &=
  -\frac12
  \sum_{I_1 \dot\cup I_2= [n]}
  \int_0^t \lds \, 
    \qbb{
      \sum_{i\in I_1,j\in I_2} \ludot_s(\xi_i,\xi_j)
      \tilde V_s(\xi_{I_1})\tilde V_s(\xi_{I_2})
    } e^{-(W_t(\xi)-W_s(\xi))}
  .
\end{align}

\subsection{Proof of Proposition~\ref{prop:Vtbd-thres1-2}: the term $n=2$}

The following two lemmas give the explicit form of $\tilde V(\xi_1,\xi_2)$
and bounds on the heat kernel that imply the required bound.

\begin{lemma}
\begin{equation} \label{e:V2}
  \tilde V_t(\xi_1,\xi_2)
  =
  -z_t^2 (1-e^{-\beta \sigma_1\sigma_2 C_t(x_1,x_2)}).
\end{equation}
\end{lemma}

\begin{proof}
By \eqref{e:polchinski-fourier-duhamel} and using that $V_s(\xi) = z_s = z_0e^{-\frac{\beta}{2} C_s(0,0)}$ by \eqref{e:V1},
\begin{align}
  \tilde V_t(\xi_1,\xi_2)
  &=
  - \int_0^t ds \,
  \dot u_s(\xi_1,\xi_2)
    \tilde V_s(\xi_{1})  \tilde V_s(\xi_{2})
    e^{-(W_t(\xi_1,\xi_2)-W_s(\xi_1,\xi_2))}
    \nnb
  &=
    - z_0^2
    e^{-W_t(\xi_1,\xi_2)}
    \int_0^t ds \,
    \dot u_s(\xi_1,\xi_2)
    e^{- \beta C_s(0,0)}
    e^{W_s(\xi_1,\xi_2)}
    .
\end{align}
Let $\sigma =\sigma_1\sigma_2$.
By \eqref{e:Wdef}, $-\beta C_s(0,0)+W_s(\xi_1,\xi_2) = \sigma  \beta C_s(x_1,x_2)$,
so the integral can be evaluated as
\begin{equation}
  \int_0^t ds \,
  \dot u_s(\xi_1,\xi_2)
  e^{-\beta C_s(0,0)}
  e^{W_s(\xi_1,\xi_2)}
  =
  \int_0^t ds \,
  \beta \sigma \dot C_s(x_1,x_2)
  e^{\beta \sigma C_s(x_1,x_2)}
  =
  e^{\beta \sigma C_t(x_1,x_2)}-1
  ,
\end{equation}
which after rearranging gives
\begin{equation}
  \tilde V_t(\xi_1,\xi_2)
  =
  -z_0^2 e^{-\beta C_t(0,0)-\beta \sigma C_t(x_1,x_2)} (  e^{\beta\sigma C_t(x_1,x_2)}-1)
  =
  -z_t^2 (1-e^{-\beta \sigma C_t(x_1,x_2)}).
  \qedhere
\end{equation}
\end{proof}

\begin{lemma} \label{lem:dotC-more}
  Let $U_t(x,y) = e^{\beta C_t(x,y)}-1$.
  The following bounds hold for $t \geq 0$, $f:\Lambda\to\R$, $\beta<8\pi$:
  \begin{align}
    \label{e:thresh1-Cbd1}
    \sup_{x_1}\sum_{x_2} |1-e^{-\beta C_t(x_1,x_2)}| &= O_\beta(\ell_t^2)
    \\
    \label{e:thresh1-Cbd2new}
    \sum_{x_1,x_2} |U_t(x_1,x_2)|(\lQ_tf(x_1) - \lQ_tf(x_2))^2 &= O_\beta(\ell_t^4 \vartheta_t^2)|f|_2^2
  \end{align}
  and again analogous estimates hold in the conservative case.
\end{lemma}

\begin{proof}
  The lemma again follows from estimates for the heat kernel and is given in Appendix~\ref{app:dotC}.
\end{proof}

\begin{proof}[Proof of Proposition~\ref{prop:Vtbd-thres1-2}]
We first consider $V^{(2,\pm)}$. By \eqref{e:V2} and \eqref{e:thresh1-Cbd1},
\begin{equation}
  \sum_y |\tilde V_t((x,+1),(y,+1))|
  = O(|z_t|^2)   \sum_y |1-e^{-\beta C_t(x,y)}| = O(|z_t|^2\ell_t^2),
\end{equation}
which is analogous to the bound for $\beta < 4\pi$ and thus gives 
\begin{equation} \label{e:HeV2pm}
  |(\lQ_tf,\He V_t^{(2,\pm)}(\varphi) \lQ_tf)|
  = O_\beta(|z_t|^2 \ell_t^4 \vartheta_t^2) |f|_2^2
  = O_\beta(|\lz_t|^2 \vartheta_t^2) |f|_2^2
\end{equation}
exactly as in \eqref{e:HessCbd}.
On the other hand, the neutral contribution to $V^{(2)}$ is given by
\begin{equation}
  V^{(2,0)}_t(\varphi) = z_t^2 \sum_{x,y} U_t(x,y) \cos(\sqrt{\beta}\varphi_{x}-\sqrt{\beta}\varphi_{y}),
  \quad
  U_t(x,y) = e^{\beta C_t(x,y)}-1
  .
\end{equation}
Therefore
\begin{equation} \label{e:pf-He-V2-bd}
  (\lQ_t f,\He V^{(2,0)}_t(\varphi) \lQ_tf)
  =
  -z_t^2 \beta \sum_{x,y}
  U_t(x,y)
  \cos(\sqrt{\beta}\varphi_x-\sqrt{\beta}\varphi_y)
  (\lQ_tf(x)-\lQ_tf(y))^2
    .
  \end{equation}
  By \eqref{e:thresh1-Cbd2new}, the right-hand side is bounded by
  $O_\beta (|z_t|^2\ell_t^4\vartheta_t^2) |f|_2^2= O_\beta (|\lz_t|^2\vartheta_t^2) |f|_2^2$.
\end{proof}

\begin{remark}
  Similarly as in \eqref{e:thresh1-Cbd2new}, for $t>0$, $f:\Lambda\to\R$, $\beta<6\pi$,
  assuming \eqref{e:assLm}, we have
  \begin{equation} \label{e:thresh1-Cbd4}
    \sum_{x_1,x_2} |U_t(x_1,x_2)||Q_tf(x_1) - Q_tf(x_2)| = O_\beta(\ell_t^2 \vartheta_t)|f|_1;
  \end{equation}
  see Appendix~\ref{app:dotC}.
  Therefore, as in \eqref{e:pf-He-V2-bd},
  \begin{align}
  (Q_tf,\nabla V_t^{(2,0)})
    &= - z_t^2\sqrt{\beta} \sum_{x,y} U_t(x,y) \sin(\sqrt{\beta} \varphi_x-\sqrt{\beta}\varphi_y) (Q_tf(x)-Q_tf(y))
      \nnb
    &= O_\beta(|z_t|^2 \ell_t^2\vartheta_t)|f|_1
      = O_\beta(|\lz_t z_t| \vartheta_t)|f|_1
      = O_\beta(|z_t| \vartheta_t)|f|_1,
  \end{align}
  provided that $\lz_t = O(1)$. Exactly as in \eqref{e:HeV2pm}, the same bound holds for $V^{(2,\pm)}$,
  and as in \eqref{e:Vtbd-thresh1-He} for $V-V^{(2)}$.
  In summary, whenever $|\lz_t|$ is sufficiently small and \eqref{e:assLm} holds,
  \begin{equation} \label{e:QtpartialVtbd}
    \max_x |(Q_t\nabla V_t)_x| = O_\beta(|z_t|\vartheta_t).
  \end{equation}
\end{remark}

\subsection{Proof of Proposition~\ref{prop:Vtbd-thres1-no2}: the terms $n>2$}
\label{sec:6pi-ngt2}

To bound the contributions due to \eqref{e:V2}, we need the following bounds on the heat kernel.
For the statement of the bounds, we set
\begin{align}
  \label{e:delta12-def}
  \delta_{12}\lCdot_t(x_1,x_2,x_3) &= \lCdot_t(x_1,x_3)- \lCdot_t(x_2,x_3)
  \\
    \label{e:delta1234-def}
  \delta_{34}\delta_{12}\lCdot_t(x_1,x_2,x_3,x_4) &= (\lCdot_t(x_1,x_3) - \lCdot_t(x_2,x_3)) - (\lCdot_t(x_1,x_4) - \lCdot_t(x_2,x_4))
                                                    .
\end{align}

\begin{lemma} \label{lem:dotC-more2}
Let $U_t(x,y) = e^{\beta C_t(x,y)}-1$.
The following bounds hold for $t \geq 0$, $\beta<6\pi$:
\begin{align}
  \label{e:thresh1-Cbd3}
  \sup_{x_1} \sum_{x_2,x_3} |U_t(x_1,x_2)\delta_{12}\lCdot_t(x_1,x_2,x_3)|
  &= O_\beta(\ell_t^4 \vartheta_t^2)
  \\
  \label{e:thresh1-Cbd2b}
  \sup_{x_1}\sum_{x_2,x_3,x_4}|U_t(x_1,x_2)U_t(x_3,x_4)\delta_{34}\delta_{12}\lCdot_t(x_1,x_2,x_3,x_4)|
  &= O_\beta(\ell_t^6\vartheta_t^2)
    ,
\end{align}
and the same bounds hold with the roles of the $x_i$ exchanged.
Also, for all $t>s>0$, $x_i\in\Lambda$, 
\begin{equation}
  \label{e:Ctsdiffbd}
  (C_t-C_s)(0,0) - (C_t-C_s)(x_1,x_2) + (C_t-C_s)(x_1,x_3) - (C_t-C_s)(x_2,x_3)
  \geq
  - O(1).
\end{equation}
Again analogous estimates hold in the conservative case.
\end{lemma}

\begin{proof}
  The lemma again follows from estimates for the heat kernel and is given in Appendix~\ref{app:dotC}.
\end{proof}

\begin{lemma} \label{lem:V3}
  Let $\beta < 6\pi$. Then $\nVt{3} \lesssim |z_t|^3 \ell_t^4$.
  Analogous bounds hold in the conservative case.
\end{lemma}

\begin{proof}
We start from \eqref{e:polchinski-fourier-duhamel-bis3}.
We assume $I_1=\{1,2\}$, $I_2=\{3\}$ since the other cases are analogous.
We first consider the case that $\xi_{I_1}$ is neutral. Then
\begin{align} \label{e:V3-pf1}
  &- \int_0^t ds \,
    \sum_{i=1,2} \dot u_s(\xi_i,\xi_3)
    \tilde V_s(\xi_1,\xi_2) \tilde V_s(\xi_{3}) 
    e^{-(W_t(\xi_1,\xi_2,\xi_3)-W_s(\xi_1,\xi_2,\xi_3))}
    \nnb
  &=
    \pm \beta \int_0^t \lds \,
    (\lCdot_s(x_1,x_3)-\lCdot_s(x_2,x_3)) U_s(x_1,x_2)
    z_s^3
    e^{-(W_t(\xi_1,\xi_2,\xi_3)-W_s(\xi_1,\xi_2,\xi_3))}
    .
\end{align}
By the definition of $W$ in \eqref{e:Wdef} and by \eqref{e:Ctsdiffbd},
\begin{equation} \label{e:u3bd}
  W_t(\xi_1,\xi_2,\xi_3)-W_s(\xi_1,\xi_2,\xi_3)
  \geq \frac{\beta}{2} (C_t-C_s)(0,0) - O(1)
  = \frac{\beta}{4\pi} \log(\ell_t/\ell_s)
  - O(1)
  .
\end{equation}
By \eqref{e:thresh1-Cbd3},
\begin{equation}
  \sup_{x_1}\sum_{x_2,x_3} |\delta_{12}\lCdot_s(x_1,x_2,x_3) U_s(x_1,x_2)| \lesssim \ell_s^4 \vartheta_s^2
  .
\end{equation}
Substituting these bounds into \eqref{e:V3-pf1},
this shows that the contribution to $\nVt{3}$
from neutral $\xi_{I_1}$ is bounded by
\begin{align}
\label{eq: upper bound neutral n=3}
  \ell_t^{-\beta/4\pi} \int_0^t \lds\, |z_s|^3 \ell_s^4 \ell_s^{\beta/4\pi} \, \vartheta_s^2
  \lesssim
  |z_t|^3 \ell_t^4
\end{align}
where we used \eqref{e:intzlbd-improved}.

We turn now to the charged case $\sigma_1 = \sigma_2$.
Note that \eqref{e:u3bd} follows as above if $\sigma_3=-\sigma_1$
and in fact holds with the better lower bound $\frac{3\beta}{4\pi} \log(\ell_t/\ell_s)-O(1)$
by positive definiteness of $C_t-C_s$ if $\sigma_3=\sigma_1$, i.e., if all charges are the same.
From the explicit form \eqref{e:V2} of $\tilde V_s(\xi_1,\xi_2)$, we thus get 
\begin{align*}
  &- \int_0^t ds \,
    \sum_{i=1,2} \dot u_s(\xi_i,\xi_3)
    \tilde V_s(\xi_1,\xi_2) \tilde V_s(\xi_{3}) 
    e^{-(W_t(\xi_1,\xi_2,\xi_3)-W_s(\xi_1,\xi_2,\xi_3))}
    \nnb
  & \lesssim 
 \beta \int_0^t \lds \,
    (\lCdot_s(x_1,x_3)+\lCdot_s(x_2,x_3)) |1-e^{-\beta C_s(x_1,x_2)}|
    |z_s|^3
    \left( \frac{\ell_s}{\ell_t} \right)^{\frac{\beta}{4\pi}} 
   .
\end{align*}
As the sum over $x_3$ can be controlled uniformly in $x_1,x_2$ by $O(\ell_t^2 \vartheta_t^2)$ thanks to \eqref{e:Cs-diag}
and then the sum over $x_2$ can be estimated by $O(\ell_t^2)$ thanks to \eqref{e:thresh1-Cbd1}, we end up with the same upper bound as 
in \eqref{eq: upper bound neutral n=3}. This completes the charged case.
\end{proof}

\begin{lemma} \label{lem:V4}
  Let $\beta < 6\pi$ and assume \eqref{e:assLm}.
  Then $\nVt{4} \lesssim |z_t|^4 \ell_t^6$.
  Analogous bounds hold in the conservative case.
\end{lemma}

\begin{proof}
We again start from \eqref{e:polchinski-fourier-duhamel-bis3}.
Up to permutation of the indices,
there are terms with $|I_1|=1$, $|I_2|=3$ and $|I_1|=|I_2|=2$.
We begin with the case $|I_1|=1$ and $|I_1|=3$.
Using that $|\ludot_s| \lesssim \ell_s^2\vartheta_s^2$ and that $\nVs{1} \lesssim |z_s|$
and $\nVs{3} \lesssim |z_s|^3 \ell_s^4$ (by \eqref{e:V1} and Lemma~\ref{lem:V3}),
\begin{equation}
  \sup_{\xi_1} \sum_{\xi_2,\dots,\xi_n}  |\ludot_s(\xi_i,\xi_j)  \tilde V_s(\xi_{I_1})V_s(\xi_{I_2})|
  \leq \|\ludot_s\| \nVs{1} \nVs{3}
  \lesssim |z_s|^4 \ell_s^6\vartheta_s^2,
\end{equation}
and we obtain  the claimed bound exactly as in the proof for $\beta<4\pi$.

In the remainder of the proof we bound the terms with $|I_1|=|I_2|=2$.
We begin with the case that $\xi_{I_1}$ and $\xi_{I_2}$ are both neutral.
Up to permutation of the indices,
we may then assume
$\xi_{I_1} = ((x_1,+1), (x_2,-1))$ and $\xi_{I_2} = ((x_3,+1), (x_4,-1))$.
By \eqref{e:V2}, 
using $\ludot_t(\xi_1,\xi_j)+\ludot_t(\xi_2,\xi_j) = \sigma_1\sigma_j(\lCdot_t(x_1,x_j)-\lCdot_t(x_2,x_j))$
and analogously for the sum over $j$,
\begin{equation}
  \sum_{i\in I_1,j\in I_2} \ludot_t(\xi_i,\xi_j)
  \tilde V_t(\xi_{I_1}) \tilde V_t(\xi_{I_2})
  =
  z_t^4 U_t(x_1,x_2)U_t(x_3,x_4)
  \delta_{34}\delta_{12}\lCdot_t(x_1,x_2,x_3,x_4).
\end{equation}
Hence, by \eqref{e:thresh1-Cbd2b} and \eqref{e:intzlbd} for $\beta<6\pi$,
\begin{equation}
  \sup_{x_1} \sum_{x_2,x_3,x_4} \int_0^t \lds \absa{ \sum_{i\in I_1,j\in I_2}\ludot_s(\xi_i,\xi_j)   \tilde V_s(\xi_{I_1}) \tilde V_s(\xi_{I_2})}
  \lesssim \int_0^t \lds \, |z_s|^4 \ell_s^6 \vartheta_s^2
  \lesssim |z_t|^4 \ell_t^6.
\end{equation}
In the case that $I_1$ is neutral and $I_2$ is charged, we similarly use
\begin{align}  \label{e:I2is2-n4}
  &\sup_{\xi_1} \sum_{\xi_2, \dots, \xi_n}
    \absa{
    \frac12 \int_0^t \lds \,
     \sum_{j\in I_2}  \qa{ \sum_{i\in I_1} \ludot_s(\xi_i,\xi_j) \tilde V_s(\xi_{I_1})1_{\sigma(\xi_{I_1})=0}}\tilde V_s(\xi_{I_2})1_{\sigma(\xi_{I_2}) \neq 0 }
  }
  \nnb
  &\leq  \beta
    \int_0^t \frac{ds}{\ell_s^2} \,
    \qbb{ \sup_{x_1} \sum_{x_2,x_3} \absa{(\lCdot_s(x_1,x_3)-\lCdot_s(x_2,x_3)) U_s(x_1,x_2)}}
    \qbb{ \sup_{\xi_3}\sum_{\xi_4} \abs{\tilde V_s(\xi_{I_2})}1_{\sigma(\xi_{I_2}) \neq 0 }}
  .
\end{align}
By \eqref{e:thresh1-Cbd3}, the first bracket is bounded by
\begin{equation}
  O_\beta(|z_t|^2\ell_t^4\vartheta_t^2).
\end{equation}
Since $\xi_{I_2}$ is charged, the contribution from $V(\xi_{I_2})$ term is bounded 
using  \eqref{e:thresh1-Cbd1} by
\begin{align} \label{e:I2bd-b}
  \sup_{\xi_3}\sum_{\xi_4} \abs{\tilde V_t(\xi_{I_2})}1_{\sigma(\xi_{I_2}) \neq 0 }
  \lesssim 
  |z_t|^2 \sup_{x_3} \sum_{x_4} |1-e^{- \beta C_t(x_3,x_4)}|
  \lesssim |z_t|^{2}\ell_t^2.
\end{align}
So altogether these contributions to \eqref{e:I2is2-n4} are again
bounded using \eqref{e:intzlbd} (and $\beta<6\pi$) by
\begin{equation}
  \int_0^t \frac{ds}{\ell_s^2} \, |z_s|^4 \ell_s^{6} \vartheta_s^2
  \lesssim |z_t|^4 \ell_t^6
  .
\end{equation}
Again the case that $\xi_{I_1}$ and $\xi_{I_2}$ are both charged is easier and analogous to the proof for $\beta<4\pi$
so omitted.
\end{proof}

\begin{lemma}
  Let $\beta < 6\pi$ and assume \eqref{e:assLm}.
  Then $\nVt{n} \leq n^{n-2}|z_t|^n(C_\beta\ell_t^2)^{n-1}$ for all $n\geq 5$.
  Analogous bounds hold in the conservative case.
\end{lemma}

\begin{proof} 
Similarly as in the proof of \eqref{e:Vtbd}, we make the inductive assumption that,
for some $n \geq 4$,
the bound \eqref{e:Vtbd-thresh1} holds for all $1\leq k\leq n$, $k \neq 2$.
By \eqref{e:V1} and Lemmas~\ref{lem:V3}-\ref{lem:V4}, the inductive assumption is
verified for $n=4$.
To advance the induction we again start from
\begin{equation} \label{e:polchinski-fourier-duhamel-bis2}
  |\tilde V_t(\xi_1,\dots,\xi_n)| \leq
  \frac12
  \sum_{I_1 \dot\cup I_2= [n]}
    \int_0^t ds \, 
  \absbb
  {
    \sum_{i\in I_1,j\in I_2} \dot u_s(\xi_i,\xi_j)
    \tilde V_s(\xi_{I_1})\tilde V_s(\xi_{I_2})
  }
  .
\end{equation}
For $|I_1|=n-k \neq 2$ and $|I_2|=k \neq 2$, we use
\begin{align}
  \sup_{\xi_1} \sum_{\xi_2,\dots,\xi_n}  |\dot u_s(\xi_i,\xi_j)  \tilde V_s(\xi_{I_1})  \tilde V_s(\xi_{I_2})|
  \leq \|\dot u_s\| \nVs{n-k} \nVs{k}
  ,
\end{align}
and bound the terms on the right-hand side using the inductive assumption.
Then exactly as in the proof for $\beta<4\pi$, i.e., of \eqref{e:Vtbd}, the result is
\begin{equation}
  \sup_{\xi_1} \sum_{\xi_2,\dots,\xi_n}
  \sum_{I_1 \dot\cup I_2 = [n] \atop |I_1|\neq 2, |I_2|\neq 2}
  \int_0^t ds\,
   \sum_{i\in I_1,j\in I_2}
  |\dot u_s(\xi_i,\xi_j)  \tilde V_s(\xi_{I_1})   \tilde  V_s(\xi_{I_2})|
  \leq n^{n-2} |z_t|^n (C_\beta \ell_t^2)^{n-1}.
\end{equation}
The terms with $|I_1|=2$ or $|I_2|=2$ require special treatment. By symmetry we may
assume that $|I_1|=2$
and that $I_1=\{1,2\}$ and $I_2=\{3,\dots,n\}$ with $n\geq 5$.
If $\xi_{I_1}$ is neutral, we use
\begin{align}  \label{e:I2is2-pf1}
  &\sup_{\xi_1} \sum_{\xi_2, \dots, \xi_n}
    \absa{
    \frac12 \int_0^t \lds \,
    \sum_{j\in I_2}  \qa{ \sum_{i\in I_1} \ludot_s(\xi_i,\xi_j) \tilde V_s(\xi_{I_1})
    1_{\sigma(\xi_{I_1})=0}}\tilde V_s(\xi_{I_2})
  }
  \nnb
  &\leq (n-2)
    \int_0^t \frac{ds}{\ell_s^2} \,
    \qbb{ \sup_{x_1} \sum_{x_2,x_3} \absa{(\lCdot_s(x_1,x_3)-\lCdot_s(x_2,x_3)) U_s(x_1,x_2)}}
    \qbb{ \sup_{\xi_3}\sum_{\xi_4,\dots,\xi_n} \abs{\tilde V_s(\xi_{I_2})}}
  .
\end{align}
By \eqref{e:thresh1-Cbd3}, the first bracket is bounded by
$O_\beta(z_t^2\ell_t^4\vartheta_t^2)$,
while for the second term involving $V(\xi_{I_2})$,
using inductive assumption for $\tilde V(\xi_{I_2})$ (note that $n-2 \geq 3$) to get
\begin{align} \label{e:I2bd}
  \sup_{\xi_3}\sum_{\xi_4,\dots,\xi_n} \abs{\tilde V_t(\xi_{I_2})}
  \leq \nVt{n-2}
  \leq (n-2)^{n-4} |z_t|^{n-2}(C_\beta \ell_t^2)^{n-3}.
\end{align}
So altogether these contributions to \eqref{e:I2is2-pf1} are bounded by
(using again \eqref{e:intzlbd} for $\beta<6\pi$),
\begin{align}
  O_\beta(1)(n-2)^{n-3}C_\beta^{n-3} \int_0^t |z_s|^n \ell_s^{2(n-1)} \vartheta_s^2 \, \lds
  \lesssim C_\beta^{-2} n^{n-4} |z_t|^n (C_\beta \ell_t^2)^{n-1}
  \leq n^{n-4} |z_t|^n (C_\beta \ell_t^2)^{n-1}
\end{align}
where in the last bound we have chosen $C_\beta$ sufficiently large (independently of $n$).
Summing over the $\binom{n}{2} \leq n^2$ choices for $I_1,I_2$ with $|I_1| = 2$ leads to the expected upper bound.
The charged case holds in the same way.
\end{proof}

\begin{proof}[Proof of Proposition~\ref{prop:Vtbd-thres1-no2}]
  The bounds \eqref{e:Vtbd-thresh1} follows by combining the previous three lemmas.
\end{proof}

\subsection{Proofs of Propositions~\ref{prop:SG-HeVC}-\ref{prop:SG-HeVC-kawasaki} without \eqref{e:small}} 
\label{sec:zlarge}

Finally, we remove the assumption \eqref{e:small}
at the cost of constants that are uniform in $\epsilon$ but not uniform in $L$.
For $t\leq t_0$, where $t_0$ is sufficiently small but of order $1/\epsilon^2$,
we can apply the same analysis as before. On the other hand, for $t\geq t_0$,
a very crude argument is sufficient to show that the Hessian of the effective
potential is bounded from below uniformly in $\epsilon$.
Our starting point for this is \eqref{e:PtHessV}, i.e.,
\begin{equation}
  \label{e:PtHessV-bis}
  (f,\He V_tf) = \PP_{t_0,t} (f, \He V_{t_0}f)
  - \pa{ \PP_{t_0,t}((f, \nabla V_{t_0})^2) -(\PP_{t_0,t}(f, \nabla V_{t_0}))^2}.
\end{equation}
The input from the previous analysis is summarised in the following lemma.

\begin{lemma}   \label{lem:Hessbdt}
Let $\beta<6\pi$.
Then there is a constant $\alpha = \alpha(\beta) > 0$ such that for all $t\geq 0$ satisfying $|\lz_t| \leq \alpha$
and \eqref{e:assLm}, the following bounds
hold uniformly in $\varphi \in X$, $f \in X$, and $x\in\Lambda$:
\begin{align}
  \label{e:Hessbdt}
  |(Q_tf,\He V_{t} Q_tf)| &\leq O_\beta(|z_t|\vartheta_{t}^2)|f|_2^2
  \\
  \label{e:gradVbd}
  |(Q_t\nabla V_t)_x| &\leq O_\beta(|z_t| \vartheta_t)  
  .
\end{align}
\end{lemma}

\begin{proof}
  For $\beta<4\pi$, these bounds follow exactly as in \eqref{e:HeVsformbd}-\eqref{e:HessCbd}.
  For $\beta<6\pi$, the bound on the Hessian is as in \eqref{e:V-2-bd} and \eqref{e:Vtbd-thresh1-He},
  and for $\nabla V_t$, see \eqref{e:QtpartialVtbd}.
\end{proof}

\begin{proof}[Proof of Theorems~\ref{prop:SG-HeVC}-\ref{prop:SG-HeVC-kawasaki} without \eqref{e:small}]
  From \eqref{e:Cs-diag},
  recall that $e^{-\frac{\beta}{2} C_t(0,0)} \asymp \ell_t^{-\beta/4\pi}$ and
  hence that $|z_t| \asymp \epsilon^{2} (\epsilon \ell_t)^{-\beta/4\pi} |z|$ and
  $|\lz_t| \asymp (\epsilon\ell_t)^{2-\beta/4\pi}|z|$.
  Here $a\asymp b$ denotes that $c_\beta \leq a/b \leq 1/c_\beta$ for some constant $c_\beta>0$.
  Let $t_\alpha>0$ be such that $|\lz_{t_\alpha}| = \alpha$.
  Thus $\epsilon \ell_{t_\alpha} \asymp (\alpha/|z|)^{1/(2-\beta/4\pi)}$ and hence
  \begin{equation}
    |z_{t_\alpha}| = O_\beta(\epsilon^2 (\epsilon\ell_{t_\alpha})^{-\beta/4\pi}|z|)
    = O_\beta(\epsilon^2 |z|^{1/(1-\beta/8\pi)}).
  \end{equation}
  Also, with $t_{m,L} = \epsilon^{-2}(m^{-2}\wedge L^2)$ as in \eqref{e:assLm},
  \begin{equation}
    |z_{t_{m,L}}| = O_\beta(\epsilon^2(m^{-1}\wedge L)^{-\beta/4\pi}|z|).
  \end{equation}
  We choose $t_0= t_\alpha \wedge t_{m,L}$ so that, since $|z_t|$ in decreasing in $t$ (see \eqref{e:ztdef}),
  \begin{equation}
    |z_{t_0}|
    = O_\beta(\epsilon^2) \pa{ (m^{-1}\wedge L)^{-\beta/4\pi}|z| + |z|^{1/(1-\beta/8\pi)}}
    = O_{\beta,z,m,L}(\epsilon^2)
    .
  \end{equation}
  With this and since $|\Lambda|=\epsilon^{-2}L^2$, it follows from \eqref{e:gradVbd} that, uniformly in $\varphi$,
  \begin{equation}
    |Q_{t_0} \nabla V_{t_0}|_2^2 = \sum_{x\in\Lambda} (Q_{t_0}\nabla V_{t_0})_x^2 \leq
    O_{\beta,z,m,L}(\epsilon^2\vartheta_{t_0}^2).
  \end{equation}
  For any $t\geq t_0$, by the Cauchy-Schwarz inequality and $|Q_{t-t_0}f|_2 \leq \vartheta_{t-t_0} |f|_2$, in particular,
  \begin{equation}  \label{e:dVbdt0}
    (Q_tf, \nabla V_{t_0})^2
    \leq O_{\beta,z,m,L}(\epsilon^2\vartheta_{t_0}^2) |Q_{t-t_0}f|_2^2
    \leq O_{\beta,z,m,L}(\epsilon^2\vartheta_t^2) |f|_2^2.
  \end{equation}
  Similarly, by \eqref{e:Hessbdt}, 
  \begin{equation} \label{e:Hessbdt0}
    |(Q_tf,\He V_{t_0} Q_tf)| \leq O_\beta(z_{t_0}\vartheta_{t_0}^2)|Q_{t-t_0}f|_2^2 = O_\beta(|z|\epsilon^2\vartheta_t^2)|f|_2^2.  
  \end{equation}
  Substituting \eqref{e:dVbdt0}-\eqref{e:Hessbdt0}
  into \eqref{e:PtHessV-bis}, using that $\PP_{t_0,t}$ is a Markov operator, we conclude that, for all $t\geq t_0$,
  \begin{equation}
    (Q_tf,\He V_{t} Q_tf) \geq \dot \mu_t  |f|_2^2, \qquad \text{where } \dot\mu_t \geq - O_{\beta,z,m,L}(\epsilon^2\vartheta_t^2).
  \end{equation}
  For $t\leq t_0$, we have $\dot\mu_t = O_\beta(|z_t|\vartheta_t^2) = O_\beta(|z|)\epsilon^2\vartheta_t^2$
  exactly as in the proofs of the theorems in the case \eqref{e:small}.
  In summary, for all $t \geq 0$,
  \begin{equation}
    \mu_{t} \geq - (O_\beta(|z|) + O_{\beta,z,m,L}(1))\int_{0}^\infty \epsilon^2 \vartheta_t^2 \geq -\mu^*(\beta,z,m,L),
  \end{equation}
  with $\mu^*(\beta,z,m,L)$ independent of $\epsilon$.
  From this bound, the remainder of the proof is the same as in the case \eqref{e:small}.
\end{proof}

\appendix
\section{Heat kernel estimates: proof of Lemmas~\ref{lem:dotC} and \ref{lem:dotC-more}-\ref{lem:dotC-more2}}
\label{app:dotC}

In this appendix, we prove Lemmas~\ref{lem:dotC} and \ref{lem:dotC-more}-\ref{lem:dotC-more2}.
These follow from standard estimates for the lattice heat kernel $p_t(x) = e^{t\Delta}(0,x)$ on $\Z^d$
and its torus version $p_t^L(x) = \sum_{y \in \Z^d} p_t(x+Ly)$, where $L \in \N$.
Throughout the appendix, $\Delta$ and $\nabla$ denote the lattice Laplacian and derivative on $\Z^d$,
not the Laplacian and gradient on $\R^\Lambda$.

\subsection{Bounds on the heat kernel}

We begin by collecting estimates on the heat kernel on $\Z^d$.
To state these, let $\alpha$ be a sequence of $|\alpha|\equiv k$ unit vectors $\alpha_1, \dots, \alpha_{k}$ in $\Z^d$,
i.e., $\alpha_i \in \{e_{1\pm}, \dots, e_{d\pm}\}$
is one of the $2d$ unit vectors $e_{i\pm}$ in $\Z^d$,
and write $\nabla^\alpha = \prod_{i=1}^k \nabla_{\alpha_i}$
with $\nabla_ef(x) = f(x+e)-f(x)$ the lattice gradient.
For $x\in\Z^d$, $|x|$ denotes any fixed norm unless stated.

\begin{lemma} \label{lem:pt}
The heat kernel $p_t$ on $\Z^d$ satisfies the following upper bounds for $t \geq 1$, $x\in\Z^d$,
and all sequences of unit vectors $\alpha$:
\begin{equation}
  \label{e:ptbounds}
  |\nabla^\alpha p_t(x)| = O_\alpha(t^{-d/2-|\alpha|/2}e^{-c|x|/\sqrt{t}}),
\end{equation}
  as well as the following asymptotics if $d=2$, for $t\geq 1$ and $x \neq 0$,
\begin{equation}  \label{e:ptasymp}
  p_t(0) = \frac{1}{4\pi t} + O(\frac{1}{t^2}),
  \qquad
  \int_0^t (p_s(0)-p_s(x)) \, ds = \frac{1}{2\pi} \log (|x| \wedge \sqrt{t}) + O(1).
\end{equation}
Moreover, the heat kernel $p_t^L$ on a discrete torus of side length $L$ satisfies, for $t\geq 1$,  $|x|_\infty <L/2$,
\begin{equation} 
  \label{e:pttorus}
  \nabla^\alpha p_t^L(x) = \nabla^\alpha p_t(x)
  + O_\alpha(t^{-|\alpha|/2}  L^{-d} e^{-cL/\sqrt{t}} )
\end{equation}
and the mean $0$ heat kernel on the torus is given by $p_t^{0,L}(x) = p^L_t(x)-1/L^2$.
\end{lemma}

\begin{proof}
Writing $\alpha_i=e_{j\sigma_j}$ with $j\in\{1,\dots, d\}$ and $\sigma_j\in\{\pm\}$
for each $i \in \{1, \dots, |\alpha|\}$,
the bound \eqref{e:ptbounds} can be seen by writing $\nabla^\alpha p_t(x)$ in its
Fourier representation:
\begin{align} \label{e:ptfourier}
  t^{d/2+|\alpha|/2} \nabla^\alpha p_t(x\sqrt{t})
  &= \frac{1}{(2\pi)^d} \int_{[-\pi,\pi]^d} \prod_{i=1}^{|\alpha|} \sqrt{t}(1-e^{i\sigma_{\alpha_i}k_{\alpha_i}})  
  e^{t \sum_{j = 1}^d (2\cos(k_j)-2)} \, e^{ikx \sqrt{t}}\, t^{d/2} \, dk
  \nnb
  &= \frac{1}{(2\pi)^d} \int_{[-t\pi,t\pi]^d} \prod_{i=1}^{|\alpha|} \sqrt{t}(1-e^{i\sigma_{\alpha_i} k_{\alpha_i}/\sqrt{t}})  e^{t\sum_{j =1}^d (2\cos(k_j/\sqrt{t})-2)} \, e^{ikx} \, dk.
\end{align}
For $t\geq 1$, the integrand is analytic on a strip $k\in (\R+i[-c,c])^d$ with $c>0$
independent of~$t$,
and hence \eqref{e:ptfourier} decays exponentially in $|x|$
(see, e.g., \cite[Chapter I.4, Exercise 4]{MR2039503}).
The first estimate in \eqref{e:ptasymp} is
standard and straightforward to verify by writing the left-hand side in terms of the Fourier transform; we thus omit its proof.
The second estimate in \eqref{e:ptasymp} is similarly standard if $t=\infty$ in which case the left-hand side is the Green function
of the discrete Laplacian:
\begin{equation} \label{e:ptasympinfty}
  \int_0^\infty (p_s(0)-p_s(x)) \, ds
  =
  \frac{1}{2\pi} \log |x| + O(1).
\end{equation}
This estimate can be found, for example, in
\cite[page 198]{MR1175176} or \cite[Theorem~4.4.4]{MR2677157} (with normalisation there differing by a factor $2d=4$).
To prove the second estimate in \eqref{e:ptasymp}
for $0<|x| \leq \sqrt{t}$, we use that by \eqref{e:ptbounds} with $|\alpha|=1$,
\begin{equation}
  \int_t^\infty (p_s(0)-p_s(x)) \, ds
  =
  O(|x|) \int_t^\infty s^{-3/2} \, ds
  = O(|x|/\sqrt{t}),
\end{equation}
which using \eqref{e:ptasympinfty} implies
\begin{equation}
  \int_0^t (p_s(0)-p_s(x)) \, ds =
  \int_0^\infty (p_s(0)-p_s(x)) \, ds  + O(|x|/\sqrt{t})
  = \frac{1}{2\pi} \log |x| + O(1).
\end{equation}
For $|x| \geq \sqrt{t}$, we use that the first bound in \eqref{e:ptasymp}
(and $p_t(0) \leq 1$ for $t<1$) implies
\begin{equation}
  \int_0^t p_s(0) \, ds = \frac{1}{2\pi} \log \sqrt{t} + O(1),
\end{equation}
and hence with \eqref{e:ptbounds} to bound $p_s(x)$,
\begin{equation}
  \int_0^t (p_s(0)-p_s(x)) \, ds
  =
  \frac{1}{2\pi} \log \sqrt{t} + O(1) - \int_1^t O(s^{-1}e^{-c|x|/\sqrt{s}}) \, ds
\end{equation}
where the integral is bounded by a multiple of
\begin{equation}
  \int_1^{t} e^{-|x|/\sqrt{s}} \,\frac{ds}{s}
  = \int_{1/|x|^2}^{t/|x|^2} e^{-1/\sqrt{s}} \,\frac{ds}{s}
  \leq \int_{0}^{1} e^{-1/\sqrt{s}} \,\frac{ds}{s} = O(1).
\end{equation}
This completes the proof of  \eqref{e:ptasymp}.

For the torus of side length $L$, we use that $p_t^L(x)= \sum_{y\in \Z^d} p_t(x+Ly)$
and set $|x|_L = \inf_{y\in \Z^d} |x+Ly|$. Then
\begin{equation}
  \sum_{y\in \Z^d} e^{-c|x+Ly|/\sqrt{t}} = e^{-c|x|_L/\sqrt{t}} + O((\sqrt{t}/L)^d e^{-\frac12 cL/\sqrt{t}}),
\end{equation}
since the remainder between the left-hand side and the first term on the right-hand side of the last equation
can be controlled by (approximating the sum by an integral and using polar coordinates)
\begin{equation}
  \int_1^\infty e^{-crL/\sqrt{t}}r^{d-1} \, dr
  \leq e^{-\frac12 c L/\sqrt{t}}\int_1^\infty e^{-\frac12 c rL/\sqrt{t}}r^{d-1} \, dr
  \leq e^{-\frac12 c L/\sqrt{t}}  (\sqrt{t}/L)^d \int_1^\infty e^{-\frac12 c r}r^{d-1} \, dr
  .
\end{equation}
This shows the estimates \eqref{e:pttorus}.

The expression for the mean $0$ heat kernel follows
from $p_t^{0,L}(x) = (\delta_0, Pe^{\Delta t} P\delta_x) = (\delta_0 - 1/L^2, e^{\Delta t} (\delta_x-1/L^2))
= p_t^L(x)-2/L^2+1/L^2 = p_t^L(x)-1/L^2$ with the projection $P$ from \eqref{e:Cinfty0}.
\end{proof}

\subsection{Proof of Lemma~\ref{lem:dotC}}

We recall the definition $\dot C_t(x) = p_t^{L_\epsilon}(x) e^{ - \epsilon^2m^2t} = p_t^{L_\epsilon}(x)  \vartheta_t^2$.
Lemma~\ref{lem:dotC} is an elementary combination of the estimates from Lemma~\ref{lem:pt},
whose details are given as follows.

\begin{proof}[Proof of Lemma~\ref{lem:dotC}]
  Applying \eqref{e:ptbounds} and \eqref{e:pttorus} with $x=0$
  to the torus of side length $L_\epsilon = L/\epsilon$ and, for $t \geq 1$, we have
  \begin{equation}
    |p_t(0)-p_t^{L_\epsilon}(0)|
    \lesssim L_\epsilon^{-d}e^{-cL_\epsilon/\sqrt{t}},
    \qquad
    p_t^{L_\epsilon}(0) \lesssim t^{-d/2} \vee L_\epsilon^{-d}.
  \end{equation}
  By the assumption \eqref{e:assLm}, either $t \leq 1/\epsilon^2m^2$ or $Lm\geq 1$ holds.
  By the above bound, if $Lm \geq 1$,
  the contribution to $C_t(0)$ from $t \geq 1/\epsilon^2m^2$ is negligible since
  \begin{align}
    \int_{1/\epsilon^2m^2}^\infty p_t^{L_\epsilon}(0) \, e^{-\epsilon^2m^2 t}\, dt
    &\lesssim
    \int_{1/\epsilon^2m^2}^\infty (t^{-1} \vee \epsilon^2 L^{-2}) \, e^{-\epsilon^2m^2 t} \, dt
    \nnb
    &\lesssim
    \epsilon^2m^2 \int_{1/\epsilon^2m^2}^\infty \, e^{-\epsilon^2m^2 t} \, dt
    \lesssim 1.
  \end{align}
  For $t \leq L^{2}/\epsilon^2$ (and thus for $t \leq 1/m^2\epsilon^2$ when $Lm \geq 1$),
  we may moreover replace $p_t^{L_\epsilon}$ by $p_t$ since
  \begin{equation}
    \int_0^{t} (p_s(0)-p_s^{L_\epsilon}(0)) \, ds
    = O(L_\epsilon^{-2}t)
    = O(1).
  \end{equation}
  Finally, the contribution to $\dot C_t(0)$ from the infinite volume heat kernel $p_t(0)$ is
  \begin{equation}
    p_t(0) e^{-\epsilon^2m^2 t}
    = [\frac{1}{4\pi t} + O(\frac{1}{t^2})] e^{-\epsilon^2m^2 t}
    = \frac{1}{4\pi t} + O(\frac{1}{t^2}) + O(\epsilon^2m^2 t),
  \end{equation}
  which integrated up to $t \leq 1/\epsilon^2m^2$ gives the main contribution
  \begin{equation}
    C_t(0) = \int_0^t p_s(0) e^{-\epsilon^2m^2 s} \, ds + O(1)
    = \frac{1}{4\pi} \log t + O(1)
    = \frac{1}{2\pi} \log \ell_t + O(1).
  \end{equation}
  This shows the first estimate in \eqref{e:Cs-diag}. The second estimate is straightforward 
  since $\lCdot_s(x,y) = \lCdot_s(0,x-y) \geq 0$ and the fact that the heat kernel defines a probability density immediately imply
  \begin{equation}
    \sup_x \sum_y \lCdot_t(x,y) =\ell_t^2  \vartheta_t^2 \sum_{y\in\Lambda} p_t^L(y)
    = \ell_t^2\vartheta_t^2 \sum_{y\in\Z^2} p_t(y)  =\ell_t^2 \vartheta_t^2. 
  \end{equation}

  Finally, in the conservative case the estimates are unchanged since
  \begin{equation}
    C_t^0(0,0) = C_t(0,0) - \frac{1}{|\Lambda|} \int_0^t e^{-\epsilon^2m^2 s}\, ds
    = C_t(0,0) - \frac{1-e^{-\epsilon^2m^2 t}}{L^2m^2}
    = C_t(0,0) + O(1)
  \end{equation}
  and
  \begin{equation}
    \sum_x |\lCdot_t^0(0,x)| \leq \sum_x (\lCdot_t(0,x) + \frac{\ell_t^2\vartheta_t^2}{|\Lambda|}) = O(\ell_t^2\vartheta_t^2).\qedhere
  \end{equation}
\end{proof}

\subsection{Proof of Lemmas~\ref{lem:dotC-more}-\ref{lem:dotC-more2}}

To prepare for the proofs of the lemmas,
we state the following consequences of Lemma~\ref{lem:pt}
in the notation used in the lemmas.
In particular, recall \eqref{e:delta12-def}-\eqref{e:delta1234-def}.
For $x\in\Lambda$, abusing notation slightly, we write $|x|$ for the
torus distance $|x|_{L_\epsilon} = \inf_{y\in\Z^d} |x+L_\epsilon y|$.
In particular, $|x| = O(L_\epsilon)$ for all $x\in\Lambda$.
Moreover, in all of the following lemmas, we impose the assumption \eqref{e:assLm}
without stating it explicitly.

\begin{lemma} \label{lem:pt-delta}
  The following estimates hold for $\lCdot_t$, $C_t$ for $t\geq 1$ and $|x-y| \geq 1$:
  \begin{equation}
    \label{e:Cxy}
    C_t(x,y) = -\frac{1}{2\pi}\log(|x-y|/\ell_t \wedge 1) + O(1),
    \qquad
    |\lCdot_t(x,y)| \lesssim \vartheta_t^2 e^{-c|x-y|/\ell_t}.
  \end{equation}
  The first bounds also implies that
  \begin{equation}
    \label{e:Cxy-cont}
    C_t(x,y) = \int_1^t \frac{1}{4\pi s} e^{-|x-y|^2/2s} e^{-\epsilon^2m^2 s} \, ds + O(1).
  \end{equation}
  For any $c'>0$ small enough,
  \begin{align}
    \label{e:pt-delta1}
    |\delta_{12}\lCdot_t(x,y,z)| e^{-c'|x-y|/\ell_t} &\lesssim \vartheta_t^2 (|x-y|/\ell_t) 
    e^{-c'|x-z|/2\ell_t}e^{-c' |y-z|/2\ell_t}
    \\
    \label{e:pt-delta2}
    |\delta_{34}\delta_{12}\lCdot_t(x,y,w,z)| e^{-c'|x-y|/\ell_t}e^{-c'|w-z|/\ell_t} 
    &\lesssim \vartheta_t^2 (|x-y|/\ell_t)(|w-z|/\ell_t) e^{-c' |x-w|/\ell_t}.
  \end{align}
  The same estimates hold with $\lCdot_t$ replaced by $\ell_t\vartheta_t \lQ_t$,
  and if $\lCdot_t$ and $\lQ_t$ are replaced by $\lCdot_t^0$ and $\lQ_t^0$.
\end{lemma}

\begin{proof}
  The estimates \eqref{e:Cxy} follow easily from those for the heat kernel
  in \eqref{e:ptbounds}-\eqref{e:pttorus}.
  Indeed, the second bound in \eqref{e:Cxy} is a special case of \eqref{e:ptbounds} and 
  \eqref{e:pttorus}:
  \begin{equation}
    \lCdot_t (x,y) = \ell_t^2 \vartheta_t^2 p_t^{L_\epsilon}(x,y)
    \lesssim \ell_t^2 \vartheta_t^2 \left( \frac{1}{t} e^{-c|x-y|/\sqrt{t}}
      + \frac{1}{L_\epsilon^2} e^{-c L_\epsilon/\sqrt{t}} \right)
    \lesssim \vartheta_t^2e^{-c|x-y|/\sqrt{t}},
  \end{equation}
  where in the last inequality we used that $\ell_t/L_\epsilon \leq 1$ follows from
  \eqref{e:assLm} and the definition of $\ell_t$ in \eqref{e:elldef}.
  Indeed, by \eqref{e:assLm}, either $t \leq L_\epsilon^2$ which implies
  $\ell_t \leq L_\epsilon$, or otherwise $Lm\geq 1$ and then also
  $\ell_t/L_\epsilon
  = (\sqrt{t} \wedge 1/(\epsilon m))/(L/\epsilon)
  \leq \sqrt{\epsilon^2 m^2 t} \wedge 1 \leq 1$.
  
  For the first bound in \eqref{e:Cxy} we note that \eqref{e:ptasymp} implies
  \begin{equation}
    \int_0^t p_s(x) \, ds
    = \frac{1}{2\pi} \qa{\log \sqrt{t} - \log(|x|\wedge \sqrt{t})} + O(1)
    = -\frac{1}{2\pi} \log(|x|/\sqrt{t}\wedge 1) + O(1).
  \end{equation}
  The additional factor $e^{-\epsilon^2m^2 s}$ multiplying $p_s(x)$ leads to the replacement of $\sqrt{t}$ by $\ell_t$ exactly as in the proof of \eqref{e:Cs-diag}.
  By an analogous calculation, the same formula holds with the discrete heat kernel replaced by the continuous one, i.e.,
  \begin{equation}
    \int_1^t \frac{1}{4\pi s} e^{-|x|^2/2s} \, ds
    = -\frac{1}{2\pi} \log(|x|/\sqrt{t}\wedge 1) + O(1),
  \end{equation}
  from which \eqref{e:Cxy-cont} follows after taking into account the additional factor $e^{-\epsilon^2m^2 s}$ as before.

  To verify \eqref{e:pt-delta1}-\eqref{e:pt-delta2},
  for $x,y\in \Z^d$, let $\gamma_{xy}$ be a path from $x$ to $y$ of length $|x-y|$
  where $|x|$ denotes the $1$-norm in this proof. Then \eqref{e:ptbounds} and \eqref{e:pttorus} imply
  \begin{align}
    |\delta_{12}p_t^{L_\epsilon}(x,y,z)|
    = |p_t^{L_\epsilon}(x,z)-p_t^{L_\epsilon}(y,z)|
    &\leq \sum_{u \in \gamma_{xy}} |\nabla p_t^{L_\epsilon}(u,z)| 
        \nnb
    &\lesssim \ell_t^{-3} \sum_{u \in \gamma_{xy}} e^{-c|u-z|/\ell_t}
      .
  \end{align}
  For $u \in \gamma_{xy}$, we have $|x-z|\leq  |x-u| + |u-z| \leq |x-y|+ |u-z|$,
  and we deduce from the symmetric estimate in $y$ that $- |u-z| \leq - |x-y| - |x-z|/2 - |y-z|/2$.
  Choosing $c'< c$, we get
  \begin{equation}
    |\delta_{12}p_t^{L_\epsilon}(x,y,z)|
    \lesssim  \ell_t^{-2} (|x-y|/\ell_t) e^{-c'|x-z|/2\ell_t} e^{-c'|y-z|/2\ell_t} e^{+c'|x-y|/\ell_t}.
  \end{equation}
  This completes \eqref{e:pt-delta1}.
  Analogously, again applying \eqref{e:ptbounds} and \eqref{e:pttorus} and choosing $c'<c$, we get
  \begin{align}
    |\delta_{34}\delta_{12}p_t^{L_\epsilon}(x,y,w,z)|
    &\leq \sum_{u\in \gamma_{xy}}\sum_{v \in \gamma_{wz}} |\nabla^2 p_t^{L_\epsilon}(u-v)|
      \nnb
    &\lesssim \ell_t^{-4} \sum_{u \in \gamma_{xy}} \sum_{v\in \gamma_{wz}} e^{-c|u-v|/\ell_t}
    \nnb
    &\lesssim  \ell_t^{-2} (|x-y|/\ell_t)(|w-z|/\ell_t) e^{-c' |x-w|/\ell_t} e^{+ c'  |x-y|/\ell_t}e^{+ c' |w-z|/\ell_t}
  \end{align}
  using that
  $|x-w| \leq |x-u| + |u-v| + |v-w| \leq |x-y| + |u-v|+|w-z|$.
\end{proof}

\begin{lemma} \label{lem:Ctsdiffbd}
For all $x,y,z \in\Lambda$, $0\leq s\leq t$,
\begin{equation} \label{e:Ctsdiffbd-bis}
  (C_t-C_s)(0,0) - (C_t-C_s)(x,y) + (C_t-C_s)(x,z) - (C_t-C_s)(y,z)
  \geq
  - O(1).
\end{equation}
\end{lemma}

\begin{proof}
It suffices to assume that $s\geq 1$.
Throughout this proof, $|x|$ denotes the Euclidean norm.
Suppose first that $|x-y| \leq |x-z| \wedge |y-z|$. We will show that
\begin{equation}
\label{eq: A28 difference}
  |(C_t-C_s)(x,z)-(C_t-C_s)(y,z)|
  \leq
  \int_s^t |\uCdot_u(x,z)-\uCdot_u(y,z)| \, du
  \lesssim 1.
\end{equation}
Indeed, this bound follows from the following two estimates:
using \eqref{e:ptbounds} with $|\alpha|=0$ for the first bound and with $|\alpha|=1$ for the second bound,
and also \eqref{e:pttorus} for the error due to periodicity,
\begin{align}
  \int_s^{|x-y|^2} (|\uCdot_u(x,z)|+|\uCdot_u(y,z)|) \, du
  &\lesssim
  1+ \int_s^{|x-y|^2} u^{-1} e^{-c|x-y|/\sqrt{u}} \, du \lesssim 1
  \\
  \int_{|x-y|^2}^{t} |\uCdot_u(x,z)-\uCdot_u(y,z)| \, du
  &\lesssim
  1+ |x-y| \int_{|x-y|^2}^t u^{-3/2} \, du \lesssim 1.
\end{align}
Here we have used that the remainder in \eqref{e:pttorus} due to the periodicity is bounded by
\begin{equation}
\frac{|x-y|}{L_\epsilon^2} \int_{|x-y|^2}^t u^{-1/2}   e^{-cL_\epsilon/\sqrt{u} - \epsilon^2m^2 u } 
\lesssim 1+ \frac{|x-y|}{L_\epsilon^2} \int_{|x-y|^2}^{\epsilon^{-2} m^{-2}} u^{-1/2}   e^{-cL_\epsilon/\sqrt{u} }
\lesssim 1
\end{equation}
when $Lm\geq 1$, and that an analogous bound holds when instead $t \leq \epsilon^{-2}(m^{-2}\wedge L^2)$.
The bound \eqref{e:Ctsdiffbd-bis} then follows from \eqref{eq: A28 difference} and
$(C_t-C_s)(0,0)- (C_t-C_s)(x,y) \geq 0$ which holds by the positive definiteness of $C_t-C_s$
and translation invariance.

The same argument as above also applies if $|y-z| \leq |x-z| \wedge |x-y|$.
Therefore suppose that $|x-z| \leq |x-y| \wedge |y-z|$.
From \eqref{e:Cxy-cont} recall that
\begin{equation}
  C_t(x,z) = \int_1^t \frac{1}{4\pi u} e^{-|x-z|^2/2u} e^{-\epsilon^2m^2 u} \, du + O(1).
\end{equation}
Since $e^{-|x-z|^2/2u} \geq e^{-|y-z|^2/2u}$ therefore
\begin{equation}
  (C_t-C_s)(x,z) - (C_t-C_s)(y,z)
  \geq -O(1).
\end{equation}
The conclusion \eqref{e:Ctsdiffbd-bis} now again follows from $(C_t-C_s)(0,0)- (C_t-C_s)(x,y) \geq 0$.
\end{proof}

\begin{lemma}
  Let $U_t(x)=e^{\beta C_t(0,x)} - 1$. Then for  $\beta < 2\pi(k+2)$ and sufficiently small $c' >0$,
  \begin{equation}
    \label{e:Usumx}
    \sum_{x} |U_t(x)| (|x|/\ell_t)^k e^{c' |x|/\sqrt{t}}
    \lesssim
    \ell_t^{2}
    .
  \end{equation}
  The analogous estimate holds in the conservative case.
\end{lemma}

\begin{proof}
  By \eqref{e:Cxy}, $C_s(0,x) = - \frac{1}{2\pi} \log(|x|/\ell_s \wedge 1) + O(1)$
  and $|\lCdot_s(0,x)| \lesssim \vartheta_s^2 e^{-c|x|/\sqrt{s}}$. Therefore
  \begin{align}
    |U_t(x)|
    = |e^{\beta C_t(0,x)}-1|
    &\leq
    \int_0^t \beta |\lCdot_s(0,x)| e^{\beta C_s(0,x)} \, \lds
    \nnb
    &\lesssim
    \int_0^t \pbb{ \ell_s^{\beta/2\pi} |x|^{-\beta/2\pi} e^{-c|x|/\sqrt{s}} e^{-\epsilon^2m^2 s} } \, \lds
    \label{e:Cbd1bdpf1}
    .
  \end{align}
  Choosing $c'< c/2$, we get
  $e^{c' |x|/\sqrt{t}} e^{-c|x|/\sqrt{s}} \leq e^{-\frac12 c|x|/\sqrt{s}}$ for $t\geq s$.
  Furthermore
  \begin{equation}
    \sum_x |x|^{k-\beta/2\pi} e^{-\frac12 c|x|/\sqrt{s}} \lesssim \sqrt{s}^{2+k-\beta/2\pi}
  \end{equation}
  holds if $2+k>\beta/2\pi$ and $s \geq 1$. Therefore
  \begin{equation}
    \sum_x |U_t(x)| (|x|/\ell_t)^k e^{c' |x|/\sqrt{t}}
    \lesssim
    \ell_t^{-k} \int_0^t \pbb{ \sqrt{s}^{2+k} e^{-\epsilon^2m^2 s} } \, \lds
    \lesssim \ell_t^2.
  \end{equation}
  The bounds are the same in the conservative case.
\end{proof}

With the above preparation, we now prove Lemmas~\ref{lem:dotC-more}-\ref{lem:dotC-more2}.

\begin{proof}[Proof of \eqref{e:thresh1-Cbd1}]
  For \eqref{e:thresh1-Cbd1}, we use $C_t(0,x)\geq 0$ which with $1-e^{-x} \leq x$ for $x\geq 0$ gives the claim
  \begin{equation}
    \sum_{x} |1-e^{-C_t(0,x)}|
    = \sum_{x} (1-e^{-C_t(0,x)})
    \leq \sum_{x} C_t(0,x) = O(\ell_t^2)
    .
  \end{equation}
  In the conservative case, $C_t^0(x) \geq -1/L^2$ and the claim follows similarly from $|1-e^{-x}| \leq 2|x|$
  for $x\geq -1$.
\end{proof}

\begin{proof}[Proof of \eqref{e:thresh1-Cbd2new}]
  For sufficiently small $c'>0$, we write
  \begin{equation}
    \sum_{x,y} |U_t(x,y)| (\lQ_t f(x)-\lQ_t f(y))^2 =
    \sum_{x,y} A_{xy} B_{xy}^2,
  \end{equation}
  where 
  \begin{align}
    A_{xy} &= |U_t(x,y)| (|x-y|/\ell_t)^2e^{2 c' |x-y|/\ell_t},
    \\
    B_{xy} &= \frac{|\lQ_t f(x)-\lQ_t f(y)|}{|x-y|/\ell_t}e^{-  c' |x-y|/\ell_t}1_{x\neq y}
    .
  \end{align}
  By \eqref{e:Usumx}, then $\sup_x\sum_y A_{xy} \lesssim \ell_t^2$ for $c'>0$ small enough.
  By \eqref{e:pt-delta1} for $\ell_t \vartheta_t \lQ_t$ instead of $\lCdot_t$
  and the inequality $2ab \leq a^2+b^2$, we have for $x\neq y$,
  \begin{align}
    \frac{|\lQ_t(x,z)-\lQ_t(y,z)|}{|x-y|/\ell_t} e^{- c' |x-y|/\ell_t}
    &\lesssim \frac{\vartheta_t}{\ell_t} e^{- c' |x-z|/2\ell_t}e^{- c'|y-z|/2\ell_t}
    \nnb
    &\leq \frac{\vartheta_t}{2\ell_t} (e^{- c' |x-z|/\ell_t}+e^{- c' |y-z|/\ell_t})
    .
  \end{align}
  Thus there are positive $M_{xy} = M_{yx} = O(\vartheta_t\ell_t^{-1}e^{- c' |x-y|/\ell_t})$,
  i.e., $\sup_x \sum_y M_{xy} \lesssim \ell_t \vartheta_t $, such that  
  \begin{equation}
    B_{xy} \leq \sum_{z} (M_{xz}+M_{yz})|f_z|.
  \end{equation}
  Then (using $(a+b)^2 \leq 2a^2+2b^2$ and $A_{xy}=A_{yx}$),
  \begin{align}
    \sum_{x,y} A_{xy} B_{xy}^2
    &\leq \sum_{x,y} A_{xy} \qa{\sum_z M_{xz}|f_z| + \sum_z M_{yz}|f_z|}^2
      \nnb
    &\leq 4\sum_{x,y} A_{xy} \qa{\sum_z M_{xz}|f_z|}^2
      \leq 4 \qa{\sup_x \sum_{y} A_{xy}} \sum_x \qa{\sum_z M_{xz}|f_z|}^2.
  \end{align}
  Similarly (with $2|ab| \leq a^2+b^2$ and $M_{xy}=M_{yx}$)
  \begin{align}
    \sum_x \qa{\sum_z M_{xz} |f_z|}^2
    &=
      \sum_{x,z,w} M_{xz}M_{xw} |f_zf_w|
  \nnb
    &\leq
      \sum_{x,z,w} M_{xz}M_{xw} |f_z|^2
      \leq
      \qa {\sup_z \sum_{x} M_{xz}}  \qa{\sup_x \sum_w M_{xw}} \sum_z |f_z|^2.
  \end{align}
  Therefore
  \begin{align}
    \sum_{x,y} A_{xy} B_{xy}^2
    \leq 4 \qa{\sup_x \sum_{y} A_{xy}}
    \qa {\sup_z \sum_{x} M_{xz}}  \qa{\sup_x \sum_w M_{xw}} |f|_2^2.
  \end{align}
  Since $\sup_x\sum_y A_{xy} \lesssim \ell_t^2$ and $\sup_x \sum_y M_{xy} \lesssim \vartheta_t\ell_t$,
  the desired bound $\lesssim \vartheta_t^2\ell_t^4$ follows.
  The bounds are unchanged in the conservative case.
\end{proof}

\begin{proof}[Proof of \eqref{e:thresh1-Cbd4}]
  We proceed analogously to the proof of \eqref{e:thresh1-Cbd2new},
  i.e., for sufficiently small $c'>0$, we write
  \begin{equation}
    \sum_{x,y} |U_t(x,y)| |\lQ_t f(x)-\lQ_t f(y)| =
    \sum_{x,y} A_{xy} B_{xy},
  \end{equation}
  where 
  \begin{align}
    A_{xy} &= |U_t(x,y)| (|x-y|/\ell_t)e^{c' |x-y|/\ell_t},
    \\
    B_{xy} &= \frac{|\lQ_t f(x)-\lQ_t f(y)|}{|x-y|/\ell_t}e^{-  c' |x-y|/\ell_t}1_{x\neq y}
    .
  \end{align}
  By \eqref{e:Usumx}, again $\sup_x\sum_y A_{xy} \lesssim \ell_t^2$ for $c'>0$ small enough,
  but now using that $\beta<6\pi$ due to the different power in the definition of $A_{xy}$.
  The bound for $B_{xy}$ is the same.
  From this, we conclude
  \begin{align}
    \sum_{x,y} A_{xy} B_{xy}
    &\leq 2\sum_{x,y} A_{xy} \qa{\sum_z M_{xz}|f_z|}
    \nnb
    &\leq 2\qa{\sup_x \sum_{y} A_{xy}} \qa{\sup_z \sum_x M_{xz}} |f|_1
    \lesssim \ell_t^3 \vartheta_t |f|_1.
  \end{align}
  Since $\lQ_t = \ell_tQ_t$, this is \eqref{e:thresh1-Cbd4}.
  The bounds are unchanged in the conservative case.
\end{proof}

\begin{proof}[Proof of \eqref{e:thresh1-Cbd3}]
  By \eqref{e:pt-delta1}
  and \eqref{e:Usumx} (with $\beta<6\pi$), one can find $c'>0$ small enough such that
  \begin{multline}
    \sup_{x_1} \sum_{x_2,x_3} |U_t(x_1,x_2)| |\delta_{12}\lCdot_t(x_1,x_2,x_3)| \\
    \lesssim
    \vartheta_t^2 \sup_{x_1} \sum_{x_2,x_3} |U_t(x_1,x_2)| e^{c'|x_1-x_2|/\ell_t}
    \frac{|x_1-x_2|}{\ell_t} e^{- c'|x_1-x_3|/2\ell_t -  c'|x_2-x_3|/2\ell_t}
    \lesssim \ell_t^4 \vartheta_t^2,
  \end{multline}
  where a factor $\ell_t^2$ comes first by summing over $x_3$ and another factor $\ell_t^2$ from \eqref{e:Usumx}.
  The same applies when the roles of $x_1,x_2,x_3$ in the $\sup$ and sum are exchanged.
  The bounds are unchanged in the conservative case.
\end{proof}

\begin{proof}[Proof of \eqref{e:thresh1-Cbd2b}]
  By \eqref{e:pt-delta2}, there is $c'>0$ small enough such that
  \begin{multline}
    |\delta_{34}\delta_{12} \lCdot_t(x_1,x_2,x_3,x_4)|e^{-c'|x_1-x_2|/\ell_t-c'|x_3-x_4|/\ell_t}
    \\
    \lesssim (|x_1-x_2|/\ell_t)(|x_3-x_4|/\ell_t) e^{- c' |x_1-x_3| / \ell_t}\vartheta_t^2,
  \end{multline}
  and using
  \eqref{e:Usumx} both for the sum over $x_2$ and $x_4$ (with $\beta<6\pi$),
  as well as the elementary bound $\sup_{x_1} \sum_{x_3} e^{-c|x_1-x_3|/\ell_t} \lesssim \ell_t^2$,
  this implies
  \begin{equation}
    \sup_{x_1}\sum_{x_2,x_3,x_4}|U_t(x_1,x_2)U_t(x_3,x_4)|
    |\delta_{34}\delta_{12} \lCdot_t(x_1,x_2,x_3,x_4)|
    \lesssim \ell_t^6 \vartheta_t^2
  \end{equation}
  with one factor $\ell_t^2$ from each of the sums.
  The bounds are unchanged in the conservative case.
\end{proof}

\section*{Acknowledgements}

We thank David Brydges for invaluable discussions and for feedback
on preliminary versions of this manuscript.
We also thank Felix Otto for pointing out the resemblence of our construction
with some works related to the Ricci flow.
We acknowledge the support of ANR-15-CE40-0020-01 grant LSD.
We would also like to thank the Isaac Newton Institute for Mathematical
Sciences for support and hospitality during the programme ``Scaling limits,
rough paths, quantum field theory'' when work on this paper was
undertaken; this work was supported by EPSRC Grant Number EP/R014604/1.

\bibliography{all}
\bibliographystyle{plain}

\end{document}